\documentclass[11pt,reqno,a4paper]{amsart}

\newcommand{\eps}{\varepsilon}
\begin{document}

\newcommand{\C}{\mathbb C}
\newcommand{\N}{\mathbb N}
\newcommand{\R}{\mathbb R}

\renewcommand{\Im}{{\mathrm{Im}}}
\renewcommand{\Re}{{\mathrm{Re}}}

\def\qqs{\ \forall \ }
\def \sprd#1,#2,#3{{\left( #1,#2\right)}_{#3}}
\newcommand{\lmd}{{\mathcal L}}

\newtheorem{theo}{Theorem}[section]
\newtheorem{prop}[theo]{Proposition}
\newtheorem{lem}[theo]{Lemma}
\newtheorem{cor}[theo]{Corollary}
\theoremstyle{definition}
\newtheorem{defi}[theo]{Definition}
\theoremstyle{remark}
\newtheorem{rem}[theo]{Remark}

\renewcommand{\cdots}{\dots}

\renewcommand{\theenumi}{\roman{enumi}}
\renewcommand{\labelenumi}{\theenumi)}

\title[The multiple tunnel effect on a
star-shaped network]{ Multiple tunnel effect for dispersive waves on
a star-shaped network: an explicit formula for the spectral
representation }

%----------Author 1
\author{F. Ali Mehmeti}

\address{%
Univ Lille Nord de France, F-59000 Lille, France \newline
\indent
 UVHC, LAMAV, FR CNRS 2956, F-59313 Valenciennes, France}

\email{felix.ali-mehmeti@univ-valenciennes.fr}

%----------Author 2
\author{R. Haller-Dintelmann}

\address{%
TU Darmstadt \\
Fachbereich Mathematik \\
Schlo{\ss}gartenstra{\ss}e 7\\
64289 Darmstadt\\
Germany}

\email{haller@mathematik.tu-darmstadt.de}

%----------Author 3
\author{V. R\'egnier}

\address{%
Univ Lille Nord de France, F-59000 Lille, France \newline
\indent
 UVHC, LAMAV, FR CNRS 2956, F-59313 Valenciennes, France}

\email{Virginie.Regnier@univ-valenciennes.fr}

\begin{abstract}
We consider the Klein-Gordon equation on a star-shaped network
composed of $n$ half-axes connected at their origins. We add a
potential which is constant but different on each branch. The
corresponding spatial operator is self-adjoint and we state explicit
expressions for its resolvent and its resolution of the identity in
terms of generalized eigenfunctions. This leads to a generalized
Fourier type inversion formula  in terms of an expansion in
generalized eigenfunctions.
%%%%%%% nouveau  %%%%%%%%%%%%%%%%%%%%%%%%%%%%%%%
Further we  prove the surjectivity of the associated transformation,
thus showing that it is in fact a spectral representation.
%%%%%%%%%%%%%%%%%%%%%%%%%%%%%%%%%%%%%%%%%%

The characteristics of the problem are marked by the non-manifold
character of the star-shaped domain. Therefore the approach via the
Sturm-Liouville theory for systems is not well-suited.
%%%%%%% nouveau  %%%%%%%%%%%%%%%%%%%%%%%%%%%%%%
The considerable effort to construct explicit formulas involving the
tunnel effect generalized eigenfunctions is justified for example by
the perspective to study the influence of tunnel effect on the
$L^{\infty}$-time decay.
%%%%%%%%%%%%%%%%%%%%%%%%%%%%%%%%%%%%%%%%%%
\end{abstract}



\thanks{Parts of this work were done, while the second author visited
the University of Valenciennes. He wishes to express his gratitude to
F. Ali Mehmeti and the LAMAV for their hospitality}

\maketitle

%
%
%%%%%%%%%%%%%%%%%%%%%%%%%%%%%%%%%%%%%%%%%%%%%%%%%%%%%%%%%%%%%%%%%%%%%%%%%%%%
%%%%%%%%%%%%%%%%%%%%%%%%%%%%%%%%%%%%%%%%%%%%%%%%%%%%%%%%%%%%%%%%%%%%%%%%%%%%
%
%
\section{Introduction}
%
%
%%%%%%%%%%%%%%%%%%%%%%%%%%%%%%%%%%%%%%%%%%%%%%%%%%%%%%%%%%%%%%%%%%%%%%%%%%%%
%%%%%%%%%%%%%%%%%%%%%%%%%%%%%%%%%%%%%%%%%%%%%%%%%%%%%%%%%%%%%%%%%%%%%%%%%%%%
%
%
\noindent This paper is motivated by the attempt to study the local
behavior of waves near a node in a network of one-dimensional media
having different dispersion properties. This leads to the study of a
star-shaped network with semi-infinite branches. Results in
experimental physics \cite{Nim, hai.nim}, theoretical physics
\cite{D-L} and functional analysis \cite{Alreg3,yas} describe new
phenomena created in this situation by the dynamics of the tunnel
effect: the delayed reflection and advanced transmission near nodes
issuing two branches. It is of major importance for the
comprehension of the vibrations of networks to understand these
phenomena near ramification nodes i.e.~nodes with at least 3
branches. The associated spectral theory induces a considerable
complexity  (as compared with the case of two branches) which is
unraveled in the present paper.

Perturbation arguments do not seem to be suited for the above
mentioned   applications, as they lead to Neumann series in the
solution formulae that hide the interesting effects.

The dynamical problem can be described as follows:

\noindent
Let  $N_1, \cdots, N_n$ be $n$ disjoint copies of $(0,+ \infty)$ ($n \geq 2$).
Consider numbers  $a_k, c_k$ satisfying $0 < c_k$, for $k = 1,
\dots, n$ and $0 \leq a_1 \leq a_2 \leq \ldots \leq a_n < + \infty$.
Find a vector $(u_1, \dots, u_n)$ of functions $u_k: [0, +\infty)
\times \overline{N_k} \rightarrow \C$ satisfying the Klein-Gordon equations
\[ [\partial_t^2 - c_k \partial_x^2   + a_k ] u_k(t,x) = 0 \ , k = 1, \dots, n,
\]
on $N_1, \cdots, N_n$ coupled at zero by usual Kirchhoff conditions
and complemented with initial conditions for the functions $u_k$ and their
derivatives.

Reformulating this as an abstract Cauchy problem, one is confronted with the
self-adjoint operator $A = (-c_k \cdot \partial^2_x + a_k)_{k=1, \dots, n}$
in $L^2(N)$, with a domain that incorporates the Kirchhoff transmission
conditions at zero. For an exact definition of $A$, we refer to
Section~\ref{sec2}.

Invoking functional calculus for this operator, the solution can be given in
terms of
\[e^{\pm i \sqrt{A}t}u_0 \hbox{ and }e^{\pm i \sqrt{A}t}v_0.
\]
The refined study of transient phenomena thus requires concrete
formulae for the spectral representation of $A$. The seemingly
straightforward idea to view this task as a Sturm-Liouville problem
for a system (following \cite{weid2}) is not well-suited, because
the resulting expansion formulae do not take into account the
non-manifold character of the star-shaped domain. The ansatz used in
\cite{weid2} inhibits the exclusive use of generalized
eigenfunctions satisfying the Kirchhoff conditions. This is proved
in Theorem \ref{comp.weidm} in the appendix of this paper, which
furnishes the comparison of the two approaches.

A first attempt to use well-suited generalized eigenfunctions in the
ramified case but without tunnel effect \cite{AHR} leads to a
transformation whose inverse formula is different on each branch.
The desired results for two branches but
with tunnel effect are implicitly included in \cite{weid2}. For $n$
branches but with the same $c_k$ and $a_k$ on all branches a variant
of the above problem has been treated in \cite{Alreg1} using Laplace
transform in $t$.

In the present paper we start by following the lines of
\cite{AHR}. In Section~\ref{sec:exp.eig}, we define $n$ families of
generalized eigenfunctions of $A$, i.e.~formal solutions
$F_\lambda^k$ for $\lambda \in [a_1, + \infty)$ of the equation
             \[AF_\lambda^k = \lambda F_\lambda^k\]
satisfying the Kirchhoff conditions in zero, such that
$e^{\pm i \sqrt{\lambda}t}F_\lambda^k(x)$ represent incoming or
outgoing plane waves on all branches except $N_k$ for $\lambda \in
[a_n, + \infty)$. For $\lambda \in [a_p, a_{p+1})$, $1\leq p <n$ we have
no propagation but exponential decay in $n-p$ branches: this
expresses what we call the multiple tunnel effect, which is new with respect to
\cite{AHR}. Using variation of constants, we derive a formula
for the kernel of the resolvent of $A$ in terms of the
$F_\lambda^k$.

Following the classical procedure, in Section~\ref{app.sto} we
derive a limiting absorption principle for $A$, and then we insert
$A$ in Stone's formula to obtain a representation of the resolution
of the identity of $A$ in terms of the generalized eigenfunctions.

The aim of the paper, attained in Section~\ref{pla.the}, is the
analysis of the Fourier type transformation
\[ (V f)(\lambda):=
\bigl( (V f)_k(\lambda)\bigr)_{k = 1, \ldots, n}
:=\Bigl( \int_N f(x) \overline{(F_{\lambda}^{ k})}(x) dx
\Bigr)_{k = 1, \ldots, n}
\]
in view of constructing its inverse. We show that it diagonalizes
$A$ and determine a metric setting in which it is an isometry. This
permits to express regularity and compatibility of $f$ in terms
of decay of $Vf$.

A major task in this context is to overcome the cyclic structure
that the cyclic nature of the $n$-star induces in the underlying
resolvent formula derived in Section~\ref{sec:exp.eig}. To this end,
we use a symmetrization procedure that is carried out in
Section~\ref{spe.rep}. It combines the expression for the resolution
of the identity $E(a,b)$ found in Section~\ref{app.sto} with an
ansatz for an expansion in generalized eigenfunctions:
\[ f(x) = \int_a^b \sum_{l,m = 1}^n q_{lm}(\lambda)
     F_\lambda^{l}(x) (Vf)_k(\lambda)  \; d \lambda.
\]
This creates a $(3 n^2 + 1) \times n^2$ linear system for the $q_{lm}$,
whose solution leads to the result in Theorem~\ref{maintheo} and to
the Plancherel type formula.

A direct approach to the same symmetrization problem, carried out in
Section~\ref{dir.app}, yields a closed formula for the matrix $q$
based on $n\times n$ matrices. This approach is to a great extent
independent of the special setting and is thus supposed to be
generalizable.

In Section~\ref{pla.the} the desired inversion formula as well as
the Plancherel type theorem are stated. Finally the domains of the
powers of $A$ are characterized using the decay properties of
$Vf$. We show that $V$ is an ordered spectral representation (see
Definition XII.3.15, p. 1216 of \cite{Dun}). The spectrum has $n$
layers and it is $p$-fold on the frequency band $[a_p,a_{p+1})$.
This reflects a kind of continuous Zeemann effect caused by the
constant, semi infinite potentials on the branches $N_j$ given by
the terms $a_j u_j$. On this frequency band the generalized
eigenfunctions have an exponential decay on $n-p$ branches,
expressing the multiple tunnel effect.

Our results are designed to serve as tools in applications
concerning the dynamics of the tunnel effect at ramification nodes.
In particular, we think of retarded reflection (following
\cite{Alreg3,hai.nim}), advanced transmission at barriers (following
\cite{Nim,D-L,yas}), $L^{\infty}$-time decay (following
\cite{fam2,fam3}), the study of more general networks of wave guides
(for example microwave networks \cite{poz}), causality and global
existence for nonlinear hyperbolic equations (following
\cite{Alreg4}) and the generalization to coupled transmission
conditions (following \cite{cm.cbc}).

Finally, let us comment on some related results.
The existing general literature on expansions in generalized
eigenfunctions (\cite{bere,weid,weid2} for example) does not seem to
be helpful for our kind of problem: their constructions start from
an abstractly given spectral representation. But in concrete cases
you do not have an explicit formula for it at the beginning.

In \cite{vbl.evl} the relation of the eigenvalues of the Laplacian
in an $L^{\infty}$-setting on infinite, locally finite networks to
the adjacency operator of the network is studied. The question of
the completeness of the corresponding eigenfunctions, viewed as
generalized eigenfunctions in an $L^2$-setting, seems to be open.

In  \cite{kostr}, the authors consider general networks with
semi-infinite ends. They give a construction to compute generalized
eigenfunctions from the coefficients of the transmission conditions
and the asymptotic behaviour of eigenvalues is studied. The
generality of the aproach, however, does not allow for explicit
inversion formulas for a given family of generalized eigenfunctions.

Spectral theory for the Laplacian on finite networks has been
studied since the 1980ies for example by J.P. Roth, J.v. Below, S.
Nicaise, F. Ali Mehmeti. A list of references can be found in \cite{fam1}.

In \cite{kra.sik} the transport operator is considered on finite
networks. The connection between the spectrum of the adjacency
matrix of the network and the (discrete) spectrum of the transport
operator is established. A generalization to infinite networks is
contained in \cite{dorn}.

For surveys on results on networks and multistructures,
cf.~\cite{Lum,exn}.

Many results have been obtained in spectral theory for elliptic
operators on various types of unbounded structures for example
\cite{werner,crd,AMMM,fam3}, cf.~ especially the references
mentioned in \cite{fam3}.

%
%%%%%%%%%%%%%%%%%%%%%%%%%%%%%%%%%%%%%%%%%%%%%%%%%%%%%%%%%%%%%%%%%%%%%%%%%%%%
%%%%%%%%%%%%%%%%%%%%%%%%%%%%%%%%%%%%%%%%%%%%%%%%%%%%%%%%%%%%%%%%%%%%%%%%%%%%
%
%
\section{Data and functional analytic framework} \label{sec2}
%
%
%%%%%%%%%%%%%%%%%%%%%%%%%%%%%%%%%%%%%%%%%%%%%%%%%%%%%%%%%%%%%%%%%%%%%%%%%%%%
%%%%%%%%%%%%%%%%%%%%%%%%%%%%%%%%%%%%%%%%%%%%%%%%%%%%%%%%%%%%%%%%%%%%%%%%%%%%
%
%
\noindent Let us introduce some notation which will be used
throughout the rest of the paper.

\medskip

\paragraph{\bf Domain and functions} Let $N_1, \cdots, N_n$
      be $n$ disjoint sets identified with $(0,+ \infty)$ ($n \in \N$,
      $n \geq 2$) and put $N :=
      \bigcup_{k=1}^n \overline{N_k}$, identifying the endpoints $0$, see
      \cite{fam1} for a detailed definition.
     Furthermore, we write
      $[a,b]_{N_k}$ for the interval $[a,b]$ in the branch $N_k$.
      For the notation of functions two viewpoints are used:
      \begin{itemize}
      \item functions $f$ on the object $N$ and $f_k$ is the restriction of
        $f$ to $N_k$.
      \item $n$-tuples of functions on the branches $N_k$; then sometimes we
            write $f = (f_1, \cdots, f_n)$.
      \end{itemize}
\paragraph{\bf Transmission conditions}
      \begin{align*}
        \strut \text{($T_0$): } & (u_k)_{k=1,\dots,n} \in
               \prod_{k=1}^n C^0(\overline{N_k}) \text{ satisfies } u_i(0) =
               u_k(0) \qqs i,k \in \{ 1, \cdots, n \}.
        \intertext{This condition in particular implies that $(u_k)_{k=1, \dots, n}$ may
            be viewed as a well-defined function on $N$.}
        \text{($T_1$): } & (u_k)_{k=1,\dots,n} \in \prod_{k=1}^n
               C^1(\overline{N_k}) \text{ satisfies } \sum_{k=1}^n c_k \cdot
               \partial_x u_k(0^+) = 0.
      \end{align*}
\paragraph{\bf Definition of the operator} Define the real Hilbert
      space $H = \prod_{k=1}^n L^2(N_k)$ with scalar product
      \[ (u,v)_H = \sum_{k=1}^n (u_k,v_k)_{L^2(N_k)}
      \]
      and the operator $A: D(A) \longrightarrow H$ by
      \[ \begin{aligned}
            D(A) &= \Bigl\{ (u_k)_{k=1, \dots, n} \in \prod_{k=1}^n H^2(N_k) :
        (u_k)_{k=1, \dots, n} \text{ satisfies } (T_0) \text{ and }
        (T_1) \Bigr\}, \\
            A((u_k)_{k=1, \cdots, n}) &= (A_ku_k)_{k=1, \cdots, n} = (-c_k
        \cdot \partial^2_x u_k + a_k u_k )_{k=1, \cdots, n}.
         \end{aligned}
      \]
      Note that, if $c_k = 1$ and $a_k=0$ for every $k \in \{1, \cdots, n \}$,
      $A$ is the Laplacian in the sense of the existing literature, cf.
    \cite{vbl.evl, kostr}.
%%%%%%%%%%%%%%%%%%%%%%%%%%%%%%%%%%%%%%%%%%%%%%%%%%%%%%%%%%%%%%%%%%%%%%%%%%%%%
\begin{prop}    \label{selfadjoint}
The operator $A: D(A) \rightarrow H$ defined above is self-adjoint
and satisfies $\sigma(A) \subset [a_1,+\infty)$.
\end{prop}
\begin{proof}
Consider the Hilbert space
\[ V = \biggl\{ (u_k)_{k=1, \dots, n} \in \prod_{k=1}^n H^1(N_k) : (u_k)_{k=1, \dots, n}
\text{ satisfies } (T_0) \biggr\}
\]
with the canonical scalar product $(\cdot, \cdot)_V$. Then the
bilinear form associated with $A +(\eps - a_1)I$ is $a_\eps : V
\times V \to \C$ with
\[ a_\eps(u,v) = \sum_{k=1}^n \bigl[ c_k (\partial_x u_k, \partial_x
v_k)_{L^2(N_k)} + (a_k + \eps - a_1)(u_k, v_k)_{L^2(N_k)} \bigr].
\]
Then clearly there is a $C > 0$ with $a_\eps(u,u) \ge C (u,u)_V$ for
all $u \in V$ and all $\eps > 0$. By partial integration one shows
that the Friedrichs extension of $(a_\eps, V, H)$ is $(A + (\eps -
a_1)I, D(A))$. Thus the operator $A + (\eps - a_1)I$ is self-adjoint
and positive. Hence $\sigma(A + (\eps - a_1)I) \subset [0, +\infty)$
for all $\eps > 0$, what implies the assertion on the spectrum.
\end{proof}

%
%
%%%%%%%%%%%%%%%%%%%%%%%%%%%%%%%%%%%%%%%%%%%%%%%%%%%%%%%%%%%%%%%%%%%%%%%%%%%%%
%%%%%%%%%%%%%%%%%%%%%%%%%%%%%%%%%%%%%%%%%%%%%%%%%%%%%%%%%%%%%%%%%%%%%%%%%%%%%
%
%
\section{Expansion in generalized eigenfunctions} \label{sec:exp.eig}
%
%
%%%%%%%%%%%%%%%%%%%%%%%%%%%%%%%%%%%%%%%%%%%%%%%%%%%%%%%%%%%%%%%%%%%%%%%%%%%%%
%%%%%%%%%%%%%%%%%%%%%%%%%%%%%%%%%%%%%%%%%%%%%%%%%%%%%%%%%%%%%%%%%%%%%%%%%%%%%
%
%
\noindent
The aim of this section is to find an explicit expression for the
kernel of the resolvent of the operator $A$ on the star-shaped
network defined in the previous section.
%
%%%%%%%%%%%%%%%%%%%%%%%%%%%%%%%%%%%%%%%%%%%%%%%%%%%%%%%%%%%%%%%%%%%%%%%%%%%%%
\begin{defi}
An element $f \in \prod_{k=1}^{n} C^{\infty}(\overline{N_k})$ is
called \emph{generalized eigenfunction} of $A$, if it satisfies
$(T_0)$, $(T_1)$ and formally the differential equation $Af =
\lambda f$ for some $\lambda \in \C$.
\end{defi}
%%%%%%%%%%%%%%%%%%%%%%%%%%%%%%%%%%%%%%%%%%%%%%%%%%%%%%%%%%%%%%%%%%%%%%%%%%%%%
%
%%%%%%%%%%%%%%%%%%%%%%%%%%%%%%%%%%%%%%%%%%%%%%%%%%%%%%%%%%%%%%%%%%%%%%%%%%%%%
\begin{lem}[Green's formula on the star-shaped network] \label{greens.formula}
Denote by $V_{l_1, \cdots, l_n}$ the subset of the network $N$ defined
by
\[ V_{l_1,...,l_n} = \{ 0 \} \cup \bigcup_{k=1}^n (0, l_k)_{N_k}.
\]
Then $u, v \in D(A)$ implies
\[ \int_{V_{l_1, \cdots, l_n}} \!\!\!\!\! u''(x) v(x) \; dx =
   \int_{V_{l_1, \cdots, l_n}} \!\!\!\!\! u(x) v''(x) \; dx - \sum_{k=1}^n
   u(l_k) v'(l_k) + \sum_{k=1}^n u'(l_k) v(l_k).
\]
\end{lem}
%%%%%%%%%%%%%%%%%%%%%%%%%%%%%%%%%%%%%%%%%%%%%%%%%%%%%%%%%%%%%%%%%%%%%%%%%%%%%%
%
%
\begin{proof}
Two successive integrations by parts are used and since both $u$ and
$v$ belong to $D(A)$, they both satisfy the transmission conditions
$(T_0)$ and $(T_1)$. So
\[ \sum_{k=1}^n u_k(0) v'_k(0)= u_1(0) \sum_{k=1}^n v'_k(0) = 0.
\]
Idem for $\sum_{k=1}^n u'_k(0) v_k(0)$.
\end{proof}
%%%
%%%
This Green formula yields now as usual an expression for the resolvent of
$A$ in terms of the generalized eigenfunctions.
%
%%%%%%%%%%%%%%%%%%%%%%%%%%%%%%%%%%%%%%%%%%%%%%%%%%%%%%%%%%%%%%%%%%%%%%%%%%%%%
\begin{prop}
\label{expr.resol} Let $\lambda \in \rho (A)$ be fixed and let
$e^{\lambda}_1$, $e^{\lambda}_2$ be generalized eigenfunctions of
$A$, such that the Wronskian $w_{1,2}^{\lambda}(x)$ satisfies for
every $x$ in $N$
\[ w_{1,2}^{\lambda}(x) = \det W(e^{\lambda}_1(x),e^{\lambda}_2(x)) =
   e^{\lambda}_1(x) \cdot (e_2^{\lambda})'(x) - (e_1^{\lambda})'(x) \cdot
   e^{\lambda}_2(x) \neq 0.
\]
If for some $k \in \{1, \dots, n \}$ we have $e_2^\lambda |_{N_k} \in
H^2(N_k)$ and $e_1^{\lambda}|_{N_m} \in H^2(N_m)$ for all $m \neq k$, then for
every $f \in H$ and $x \in N_k$
\begin{equation} \label{E}
  [R(\lambda, A) f](x) = \frac{1}{c_k w_{1,2}^{\lambda}(x)} \cdot
    \left[ \int_{(x,+\infty)_{N_k}} \hspace{-.5cm} e^{\lambda}_1(x)
    e^{\lambda}_2(x') f(x') \; dx' + \int_{N \setminus (x,+\infty)_{N_k}}
    \hspace{-.8cm} e^{\lambda}_2(x) e^{\lambda}_1(x') f(x') \; dx'
    \right].
\end{equation}
\end{prop}
%%%%%%%%%%%%%%%%%%%%%%%%%%%%%%%%%%%%%%%%%%%%%%%%%%%%%%%%%%%%%%%%%%%%%%%%%%%%%
%
\noindent Note that by the integral over $N$, we mean the sum of the integrals
over $N_k$, $k=1, \cdots, n$.
\begin{proof}
Let $\lambda \in \rho (A)$. We shall show that the integral operator
defined by the right-hand side of (\ref{E}) is a left inverse of
$\lambda I - A$. Let $u \in D(A)$ and $x \in N_k$. Then
\begin{align*}
    I_\lambda :=& \int_{(x,+\infty)_{N_k}} e^{\lambda}_1(x) e^{\lambda}_2(x')
        (\lambda I - A) u(x')   \; dx' +  \int_{N \setminus (x,+\infty)_{N_k}}
    e^{\lambda}_2(x) e^{\lambda}_1(x') (\lambda I - A) u(x') \; dx' \\
   =& \; e^{\lambda}_1(x) \lim_{l_k \to \infty}
    \int_{x}^{l_k}e^{\lambda}_2(x')(\lambda I - A) u(x') \; dx'
%
%    \\
%
%   & \qquad \quad \strut
%
%&
%
   + e^{\lambda}_2(x) \lim_{l_m \to \infty, m \not= k}
    \int_{ V_{l_1,\ldots,l_{k-1},x,l_{k+1},\ldots,l_n}}
\!\!\!\!\!  \!\!\!\!\!  \!\!\!\!\! \!\!\!\!\! \!\!\!\!\! \!\!\!\!\!
e^{\lambda}_1(x')(\lambda I - A) u(x') \; dx',
\end{align*}
due to the dominated convergence Theorem, the integrands being in $L^1(\R)$ by
the hypotheses.

We have $u \in D(A) \subset \prod_{j=1}^n H^2(N_j)$ and
\[ e^{\lambda}_2 \vert_{N_k} \in H^2(N_k), \quad e^{\lambda}_1 \vert_{N_m} \in
    H^2(N_m), \ m \neq k
\]
 by hypothesis and thus
\begin{alignat*}{2}
 \partial_x u\vert_{N_k}(x)\cdot e^{\lambda}_2 \vert_{N_k}(x)
    &\mathop{\longrightarrow}_{x \to +\infty} 0, \qquad \qquad
    u\vert_{N_k}(x)\cdot \partial_x e^{\lambda}_2 \vert_{N_k}(x)
    &&\mathop{\longrightarrow}_{x \to +\infty} 0, \\
 \partial_x u\vert_{N_m}(x)\cdot e^{\lambda}_1 \vert_{N_m}(x)
    &\mathop{\longrightarrow}_{x \to +\infty} 0, \qquad \qquad
    u\vert_{N_m}(x)\cdot \partial_x e^{\lambda}_1 \vert_{N_m}(x)
    &&\mathop{\longrightarrow}_{x \to +\infty} 0, \ m \neq k,
\end{alignat*}
all products being in some $H^2(N_j)$. Recall that
\[ \int_{a}^{b} f'' g  = \int_{a}^{b} f g'' - f(b)g'(b) + f'(b)g(b) +
    f(a)g'(a) - f'(a)g(a)
\]
for $f,g \in H^2((a,b))$.
Now Lemma~\ref{greens.formula} and $(\lambda I - A)e^{\lambda}_r = 0$ for
$r =1 ,2$ imply
\begin{align*}
  I_\lambda  &= e^{\lambda}_1(x)
    \lim_{l_k \to \infty} \Bigl[ \int_{x}^{l_k} (\lambda I - A)
    e^{\lambda}_2(x') u(x') \; dx' \\
  & \qquad \strut + c_k \Bigl( -u(l_k) \partial_x e^{\lambda}_2(l_k) +
    \partial_x u(l_k) e^{\lambda}_2(l_k) + u(x) \partial_x
    e^{\lambda}_2(x) - \partial_x u(x) e^{\lambda}_2(x) \Bigr) \Bigr] \\
  & \qquad \strut + e^{\lambda}_2(x) \Bigl[ \lim_{l_m \to \infty, m \neq k}
    \int_{V_{l_1,\ldots,l_{k-1},x,l_{k+1},\ldots,l_n}} (\lambda I - A)
    e^{\lambda}_1(x') u(x') \; dx' \\
  & \qquad \strut + \sum_{j \neq k} c_j \lim_{l_j \to \infty} \Bigl( -u(l_j)
    \partial_x e^{\lambda}_1(l_j) + \partial_x u(l_j) e^{\lambda}_1(l_j)
    \Bigr) + c_k \Bigl( -u(x) \partial_x e^{\lambda}_1(x) + \partial_x
    u(x) e^{\lambda}_1(x) \Bigr) \Bigr] \displaybreak[0]\\
  &= c_k \Bigl(  e^{\lambda}_1(x) \partial_x e^{\lambda}_ 2(x) - \partial_x
    e^{\lambda}_1(x) e^{\lambda}_2(x) \Bigr) u(x) \\
  &= c_k w^{\lambda}_{1,2}(x) u(x).
\end{align*}
Now the invertibility of $\lambda I - A$ implies the result.
\end{proof}
%
%
%
%%%%%%%%%%%%%%%%%%%%%%%%%%%%%%%%%%%%%%%%%%%%%%%%%%%%%%%%%%%%%%%%%%%%%%%%%%%%%%
\begin{defi}[Generalized eigenfunctions of $A$]  \label{gen.eig}
For $k \in \{ 1, \dots, n\}$ and $\lambda \in \C$ let
\[ \xi_k(\lambda) := \sqrt{\frac{\lambda - a_k}{c_k}} \qquad \text{and} \qquad
   s_k := - \frac{\sum_{l \neq k} c_l \xi_l(\lambda)}{c_k \xi_k(\lambda)}.
\]
Here, and in all what follows, the complex square root is chosen in such a way
that $\sqrt{r \cdot e^{i \phi}} = \sqrt{r} e^{i \phi/2}$ with $r>0$ and $\phi
\in [-\pi,\pi)$.

For $\lambda \in \C$ and $j,k \in \{ 1, \cdots, n\}$, $F_{\lambda}^{\pm,j}: N \rightarrow \C$ is defined for
$x \in \overline{N_k}$ by $F_{\lambda}^{\pm,j}(x) :=
F_{\lambda,k}^{\pm,j}(x)$ with
\[ \left\{ \begin{aligned}
       F_{\lambda,j}^{\pm,j}(x) &= \cos(\xi_j( \lambda) x ) \pm i
       s_j( \lambda)   \sin(\xi_j( \lambda) x ), & \\
      F_{\lambda,k}^{\pm,j}(x) &= \exp(\pm i \xi_k(\lambda)x), &
          \text{for } k \neq j.
   \end{aligned} \right.
\]
\end{defi}
%%%%%%%%%%%%%%%%%%%%%%%%%%%%%%%%%%%%%%%%%%%%%%%%%%%%%%%%%%%%%%%%%%%%%%%%%%%%%%
%
%
%%%%%%%%%%%%%%%%%%%%%%%%%%%%%%%%%%%%%%%%%%%%%%%%%%%%%%%%%%%%%%%%%%%%%%%%%%%%%%
\begin{rem}\label{rem.eig}
\begin{itemize}
\item $F_{\lambda}^{\pm,j}$ satisfies the transmission conditions $(T_0)$ and
      $(T_1)$ and
      formally it holds $AF_{\lambda}^{\pm,j}= \lambda
      F_{\lambda}^{\pm,j}$. Thus it is a generalized eigenfunction of $A$,
      but clearly $F_{\lambda}^{\pm,j}$ does not belong to $H$, so it is
      not a classical eigenfunction.
\item For $\Im (\lambda) \neq 0$, the function $F_{\lambda,k}^{\pm,j}$, where
      the $+$-sign (respectively $-$-sign) is chosen if $\Im (\lambda) > 0$
      (respectively $\Im (\lambda) < 0$), belongs to
      $H^2(N_k)$ for $k \neq j$. This feature is used in the formula for the
      resolvent of $A$.
\end{itemize}
\end{rem}
%%%%%%%%%%%%%%%%%%%%%%%%%%%%%%%%%%%%%%%%%%%%%%%%%%%%%%%%%%%%%%%%%%%%%%%%%%%%%%
%
%
%%%%%%%%%%%%%%%%%%%%%%%%%%%%%%%%%%%%%%%%%%%%%%%%%%%%%%%%%%%%%%%%%%%%%%%%%%%%%%
\begin{defi}[Kernel of the resolvent] \label{def.K}
Let $w:\C \to \C$ be defined by $w(\lambda) := \pm i \cdot
\sum_{j=1}^n c_j \xi_j (\lambda)$. For any $\lambda \in \C$ such
that $w(\lambda) \neq 0$, $j \in \{ 1, \cdots, n \}$, and for every
$x \in \overline{N_j}$, we define
\[ K(x,x',\lambda) = \left\{ \begin{aligned}
        \frac{1}{w(\lambda)} F^{\pm,j}_{\lambda,j}(x)
              F^{\pm,j+1}_{\lambda,j}(x'), & \text{ for } x' \in \overline{N_j},
              \, x'> x, \\
        \frac{1}{w(\lambda)} F^{\pm,j+1}_{\lambda,j}(x)
              F^{\pm,j}_{\lambda}(x'), & \text{ for } x' \in \overline{N_k}, k
              \neq j \text{ or } x'\in \overline{N_j}, \, x'< x.
        \end{aligned} \right.
\]
In the whole formula $+$ (respectively $-$) is chosen, if
$\Im(\lambda) > 0$ (respectively $\Im (\lambda) \leq 0$). Finally,
the index $j$ is to be understood modulo $n$, that is to say, if
$j=n$, then $j+1 = 1$.
\end{defi}
%%%%%%%%%%%%%%%%%%%%%%%%%%%%%%%%%%%%%%%%%%%%%%%%%%%%%%%%%%%%%%%%%%%%%%%%%%%%%%
%

%
%%%%%%%%%%%%%%%%%%%%%%  Graphic NxN %%%%%%%%%%%%%%%%%%%%%%%%%%%%%%%%%%
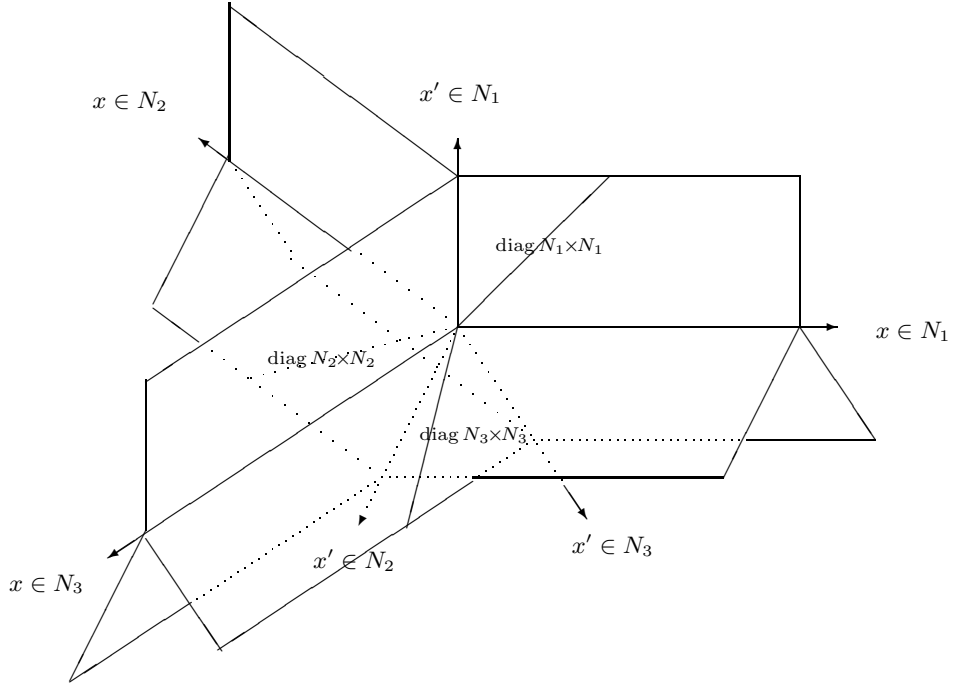
\begin{figure}

\vspace*{3cm}
\hspace*{-4cm}
\parbox{12cm}{

\unitlength=1mm \special{em:linewidth 0.4pt} \linethickness{0.4pt}
\begin{picture}(85.00,101.33)

\multiput(80.00,80)(-1.6,-0.4){18}{\rule{0.10\unitlength}{0.10\unitlength}}

\multiput(70,60.00)(-1.6,1.2){15}{\rule{0.10\unitlength}{0.10\unitlength}}

\multiput(80.00,80.00)(-1.6,1.2){9}{\rule{0.10\unitlength}{0.10\unitlength}}

\multiput(70.00,60.00)(-0.9,-0.6){28}{\rule{0.10\unitlength}{0.10\unitlength}}

\multiput(90.00,65.00)(-0.9,-0.6){9}{\rule{0.10\unitlength}{0.10\unitlength}}

\multiput(90.00,65.00)(1,0){28}{\rule{0.10\unitlength}{0.10\unitlength}}

\multiput(80.00,80.00)(-0.5,-1){25}{\rule{0.10\unitlength}{0.10\unitlength}}

\put(67.70,55.00){\vector(-1,-2){0.7}}

\multiput(50,102)(0.8,-1.2){11}{\rule{0.10\unitlength}{0.10\unitlength}}

\multiput(80.00,80.00)(0.8,-1.2){18}{\rule{0.10\unitlength}{0.10\unitlength}}

\put(94.00,59.00){\vector(2,-3){3.00}}

\multiput(90.00,65.00)(-1.6,1.2){21}{\rule{0.10\unitlength}{0.10\unitlength}}

\put(125.00,80.00){\line(0,1){20.00}}
\put(50.00,102.00){\line(0,1){21.00}}
\put(39.00,53.00){\line(0,1){20.00}}
\put(80.00,80.00){\vector(0,1){25.00}}

\put(125.00,80.00){\line(2,-3){10.00}}
\put(39.00,52.00){\line(2,-3){10.00}}

\put(125.00,80.00){\line(-1,-2){10.00}}
\put(50.00,103.00){\line(-1,-2){10.00}}
\put(39.00,53.00){\line(-1,-2){10.00}}

\put(80.00,80.00){\vector(1,0){50.00}}
\put(80.00,100.00){\line(1,0){45.00}}
\put(118.00,65.00){\line(1,0){17.00}}
\put(82.00,60.00){\line(1,0){33.00}}
\multiput(70.00,60.00)(1,0){12}{\rule{0.10\unitlength}{0.10\unitlength}}
%%%%%%%%%%%%%%

\put(66.00,90.00){\vector(-4,3){20.00}}
\put(80.00,100.00){\line(-4,3){30.00}}
\put(46.00,78.00){\line(-4,3){6.00}}

\put(80.00,80.00){\vector(-3,-2){46.00}}
\put(80.00,100.00){\line(-3,-2){41.00}}
\put(82.00,59.50){\line(-3,-2){33.50}}

\put(45.00,43.50){\line(-3,-2){16.0}}

\put(80.00,80.00){\line(-1,-4){6.67}}
\put(80.00,80.00){\line(1,1){20.00}}

\footnotesize

\put(135,78.5){$x \in N_1$}
\put(75,110){$x' \in N_1$}
\put(32,109){$x \in N_2$}
\put(21,45){$x \in N_3$}
\put(95,50){$x'\in N_3$}
\put(61,48){$x' \in N_2$}
\tiny
\put(85,90){diag$\, N_1\hspace*{-1mm}\times \hspace*{-1mm}N_1$}
\put(55,75){diag$\, N_2 \hspace*{-1mm}\times \hspace*{-1mm}N_2$}
\put(75,65){diag$\, N_3 \hspace*{-1mm}\times \hspace*{-1mm}N_3$}

\end{picture}

                   }
\vspace*{-3cm}

   \caption{$N \hspace*{-1mm}\times \hspace*{-1mm}N$ in the case $n=3$}
   \label{NxN}
\end{figure}
%%%%%%%%%%%%%%%%%%%%%%%   END GRAPHIC NxN %%%%%%%%%%%%%%%%%%%%%%%%%%%%%%%%%%%%%%
Figure~\ref{NxN} shows the domain of the kernel
$K(\cdot,\cdot,\lambda)$ in the case $n=3$ with its three main
diagonals, where the kernel is not smooth.
We will show in Theorem~\ref{theo1} that $K$ is indeed the kernel of
the resolvent of $A$. In order to do so, we collect some useful
observations in the following lemma.

Note that in particular, if $c_j = c$ and $a_j = 0$ for all $j \in
\{ 1, \dots, n \}$, then $w(\lambda) = \pm i n c \sqrt{\lambda}$ for
all $j \in \{ 1, \dots, n \}$, which only vanishes for $\lambda=0$.
On the other hand, if there exist $i$ and $j$ in $\{ 1, \ldots, n
\}$, such that $a_i \neq a_j$, then it is clear that $w(\lambda)$
never vanishes on $\R$, but we need to know, if it vanishes on $\C$.

\begin{lem} \label{w.esti}
\begin{enumerate}
\item \label{enum:w.esti:1} For $a_1 \leq \lambda$ and $\eps \geq 0$ holds
    $| w(\lambda - i \eps) |^2 \geq \sum_{j=1}^n c_j |\lambda - a_j|$.
\item \label{enum:w.esti:2} For $\lambda \in \rho(A)$ such that $\Re(\lambda)
    \geq a_1$, the Wronskian $w$ only vanishes at $\lambda = \alpha$, if
    $a_k = \alpha$ for all $k \in \{ 1, \ldots, n \}$.
\end{enumerate}
\end{lem}

\begin{proof}
We first prove \ref{enum:w.esti:1}).
Note that for $z_1,z_2,\ldots,z_n \in \C$ holds
\[ \Bigl| \sum_{j=1}^n z_j \Bigr|^2 = \sum_{j=1}^n |z_j|^2 + 2
    \sum_{{k,l=1}\atop{k\neq l}}^{n} \Re (z_k \overline{z_l}).
\]
With $z_k := c_k \xi_k(\lambda - i \eps)$ and the
abbreviation $\eta_k := \sqrt{\lambda - i \eps -a_k}$ it follows
\[ | w(\lambda - i \varepsilon) |^2 = \sum_{j=1}^n c_j |\eta_j|^2 + 2
    \sum_{{k,l=1}\atop{k\neq l}}^n c_k c_l \Re (\eta_k \overline{\eta_l}).
\]
Thus it suffices to show $\Re (\eta_k \overline{\eta_l}) \ge 0$ for
$k,l = 1, \ldots, n$ with $k \neq l$. With our convention $\sqrt{z}
= \sqrt{|z|} e^{i \frac{\arg(z)}{2}}, \ \arg(z) \in [-\pi,\pi)$,
this means
\[ \arg(\eta_k \overline{\eta_l}) \in \Bigl[ -\frac{\pi}{2}, \frac{\pi}{2}
    \Bigr].
\]
Without loss of generality, let $k<l$. Then we have $a_k \leq a_l$ and there are three possible
positions of $\lambda$:
\begin{itemize}
\item $a_l \leq \lambda$:

\noindent Then for $r = k$ and for $r = l$ we have $\arg(\lambda - i\eps -
a_r) \in [-\frac{\pi}{2},0]$ and therefore
\[ \arg(\eta_r) = \frac{1}{2} \arg(\lambda - i\eps - a_r) \in \Bigl[
    -\frac{\pi}{4}, 0 \Bigr].
\]
Using $\lambda - a_k > \lambda - a_l$, this implies
\[ \arg(\eta_k \overline{\eta_l}) = \arg(\eta_k) - \arg(\eta_l) \in \Bigl[ 0,
    \frac{\pi}{4} \Bigr].
\]
\item $a_k \leq \lambda < a_l$:

\noindent Then $\arg(\lambda - i\eps - a_{k}) \in [-\frac{\pi}{2},0]$ and
therefore
\[ \arg(\eta_k) = \frac{1}{2} \arg(\lambda - i\eps - a_k) \in \Bigl[
    -\frac{\pi}{4}, 0 \Bigr].
\]
Furthermore, $\arg(\lambda - i\eps - a_l) \in [-\pi, -\frac{\pi}{2}]$, and
therefore
\[ \arg(\eta_l) = \frac{1}{2} \arg(\lambda - i\eps - a_l) \in \Bigl[
    -\frac{\pi}{2}, -\frac{\pi}{4} \Bigr].
\]
Putting everything together we get
\[ \arg(\eta_k \overline{\eta_l}) = \arg(\eta_k) - \arg(\eta_l) \in \Bigl[ 0,
    \frac{\pi}{2} \Bigr].
\]
%%%%%%%%%%%%%%%%%%%%%
%
\item $\lambda < a_k$:

\noindent In this case we get, again for $r=k$ and $r=l$,
$\arg(\lambda - i\eps - a_r) \in [-\pi,-\frac{\pi}{2}]$ and therefore
\[ \arg(\eta_r) = \frac{1}{2} \arg(\lambda - i\eps - a_r) \in \Bigl[
    -\frac{\pi}{2}, -\frac{\pi}{4} \Bigr].
\]
This yields
\[ \arg(\eta_k \overline{\eta_l}) = \arg(\eta_k) - \arg(\eta_l) \in \Bigl[
    -\frac{\pi}{4},\frac{\pi}{4} \Bigr].
\]
\end{itemize}

Thus, in all three cases we have $\arg(\eta_k \overline{\eta_l}) \in
[ -\frac{\pi}{4}, \frac{\pi}{2} ]$ and, hence, $\Re (\eta_k \overline{\eta_l})
\geq 0$ for all $k,l = 1, \ldots ,n$ with $k \neq l$.

In order to prove \ref{enum:w.esti:2}), we note that the choice of the branch cut of the complex square root has been made in such a way that
$\sqrt{\overline{\mu}} = \overline{\sqrt{\mu}}$ for all $\mu \in \C$.
Thus $w(\overline{\mu}) = \overline{w(\mu)}$ for all $\mu \in \C$. This implies that the first part of the lemma can be generalized to:
\[ |w(\mu)|^2 \geq \sum_{j=1}^n c_j |\Re(\mu) - a_j|
\]
for every $\mu$ such that $\Re(\mu) \geq a_1$. Then $w(\mu)=0$ for
such a $\mu$ implies $|\Re(\mu) - a_j| = 0$ for all $j \in \{ 1,
\ldots, n \}$. Thus all  $a_k, k= 1, \ldots, n $ have to be equal to
some number $\alpha \ge 0$. Then $|w(\mu)|^2 = \left( \sum_{k=1}^n
\sqrt{c_k} \right)^2 |\mu - \alpha|$, which only vanishes for $\mu =
\alpha$.
\end{proof}
%%%
%%%
%
%%%%%%%%%%%%%%%%%%%%%%%%%%%%%%%%%%%%%%%%%%%%%%%%%%%%%%%%%%%%%%%%%%%%%%%%%%%%%
\begin{theo} \label{theo1}
Let $f \in H$. Then, for $x \in N$ and $\lambda \in \rho(A)$ such that
$\Re(\lambda) \geq a_1$
\[ [R(\lambda, A)f](x)= \int_N K(x, x', \lambda) f(x') \; dx'.
\]
\end{theo}
%%%%%%%%%%%%%%%%%%%%%%%%%%%%%%%%%%%%%%%%%%%%%%%%%%%%%%%%%%%%%%%%%%%%%%%%%%%%%
%
\begin{proof}
In \eqref{E}, the generalized eigenfunction $e_1^\lambda$ can be chosen to be $F^{\pm,j}_{\lambda}$. Then $e_2^\lambda$ can be
$F^{\pm,l}_{\lambda}$ with any $l \neq j$ so we have chosen $j+1$ to fix the formula. The choice has been done so that the integrands lie
in $L^1(0, + \infty)$ (cf. the last item in Remark~\ref{rem.eig}). Simple calculations yield the expression for the Wronskian
$w(\lambda)$.

For $\lambda \in \rho(A)$ such that $\Re(\lambda) \geq a_1$, the Wronskian only vanishes at $\lambda = \alpha$ if $a_k = \alpha$ for all $k \in \{ 1, \ldots, n \}$
due to Lemma~\ref{w.esti}. But in this case, the Wronskian is $w(\lambda) = \left( \sum_{k=1}^n \sqrt{c_k} \right) \sqrt{\lambda - \alpha}$ and
$w(\lambda)^{-1}$ has an $L^1$-singularity at $\lambda = \alpha$.
\end{proof}
%
%
%
%%%%%%%%%%%%%%%%%%%%%%%%%%%%%%%%%%%%%%%%%%%%%%%%%%%%%%%%%%%%%%%%%%%%%%%%%%%%%%
%%%%%%%%%%%%%%%%%%%%%%%%%%%%%%%%%%%%%%%%%%%%%%%%%%%%%%%%%%%%%%%%%%%%%%%%%%%%%%
%
%
\section{Application of Stone's formula and limiting absorption
principle} \label{app.sto}
%
%
%%%%%%%%%%%%%%%%%%%%%%%%%%%%%%%%%%%%%%%%%%%%%%%%%%%%%%%%%%%%%%%%%%%%%%%%%%%%%%
%%%%%%%%%%%%%%%%%%%%%%%%%%%%%%%%%%%%%%%%%%%%%%%%%%%%%%%%%%%%%%%%%%%%%%%%%%%%%%
%
%
\noindent
Let us first recall Stone's formula (see Theorem~XII.2.11 in \cite{Dun}).
%
%%%%%%%%%%%%%%%%%%%%%%%%%%%%%%%%%%%%%%%%%%%%%%%%%%%%%%%%%%%%%%%%%%%%%%%%%%%%%%
\begin{theo} \label{stone}
Let $E$ be the resolution of the identity of a linear unbounded
self-adjoint operator $T: D(T) \rightarrow H$ in a Hilbert space $H$
(i.e. $E(a,b) = {{\mathbf{1}}}_{(a,b)}(T)$ for $(a,b) \in \R^2$,
$a<b$). Then, in the strong operator topology
\[ h(T) E(a,b) = \lim_{\delta \rightarrow 0^+} \lim_{\eps \rightarrow
   0^+} \frac{1}{2 \pi i} \int_{a + \delta}^{b - \delta} h(\lambda)
   [R(\lambda - \eps i, T) - R(\lambda + \eps i, T)] \; d \lambda
\]
for all $(a,b) \in \R^2$ with $a<b$ and for any continuous scalar function $h$
defined on the real line.
\end{theo}
%%%%%%%%%%%%%%%%%%%%%%%%%%%%%%%%%%%%%%%%%%%%%%%%%%%%%%%%%%%%%%%%%%%%%%%%%%%%%%
%
To apply this formula we need to study the behaviour of the resolvent
$R(\lambda, A)$ for $\lambda$ approaching the spectrum of $A$.
Lemma~\ref{w.esti} will be useful as well as the following results.
First we state an estimate that follows directly from the definition
of $s_j$.
%%%%%%%%%%%%%%%%%%%%%%%%%%%%%%%%%%%%%%%%%%%%%%%%%%%%%%%%%%%%%%%%%%%%%%%%%%%%%%
\begin{lem} \label{sj}
Let $\delta > 0$ be fixed. For all $a_1 \le \lambda$, $0 < \eps < \delta$ and
$j = 1,\ldots,n$ holds
\[ |s_j(\lambda - i \eps)| \le M(\lambda, \delta) := \max_{j=1,\ldots,n}
    \Bigl\{ \frac{1}{\sqrt{|\lambda - a_j|}}  \Bigr\} \sum_{k=1}^n \bigl(
    (\lambda - a_k)^2 + \delta^2\bigr)^{1/4}.
\]
\end{lem}
%%%
%%%
Note that $M(\cdot,\delta) \in
L^1_{\mathrm{loc}}([a_1,+\infty))$.
Furthermore, if $a_1 = \ldots = a_n$, then
\[ s_j(\mu) = \frac{1}{\sqrt{c_j}} \sum_{{k=1}\atop{k\neq j}}^n
    \sqrt{c_k}
\]
for all $\mu \in \C$, which means that $s_j$ is constant and we may take
\[ M(\lambda,\delta) := \max_{j=1,\ldots,n}   \Bigl\{ \frac{1}{\sqrt{c_j}}
    \sum_{{k=1}\atop{k\neq j}}^n \sqrt{c_k} \Bigr\}.
\]

%%%%%%%%%%%%%%%%%%%%%%%%%%%%%%%%%%%%%%%%%%%%%%%%%%%%%%%%%%%%%%%%%%%%%%%%%%%%%%
\begin{theo}[Limiting absorption principle for $A$] \label{lim.abs}
Let $\delta > 0$ be fixed and let $M(\lambda,\delta)$ be defined as in
Lemma~\ref{sj}. Then for all $a_1 \leq \lambda$, $0 < \eps < \delta$ and
$(x,x') \in N^2$ we have
\begin{enumerate}
\item $\lim_{\alpha \rightarrow 0} K(x,x',\lambda - i \alpha) =
      K(x,x',\lambda)$,
\item $|K(x,x',\lambda - i \eps)| \leq N(\lambda,\delta) e^{\gamma  (x + x')}$,
    where $N(\lambda,\delta) :=
    \frac{1 + M(\lambda,\delta)}{(\sum_{j=1}^n c_j |\lambda - a_j|)^{1/2}}$
    and \\
$\gamma := \max_{j=1,\ldots,n} \{ c_j^{-\frac{1}{2}} \}
      \max \{
      ((a_n-a_1)^2  + \delta^2)^{\frac{1}{4}},
      1,
       \delta
      \}.$
\end{enumerate}
\end{theo}
%%%%%%%%%%%%%%%%%%%%%%%%%%%%%%%%%%%%%%%%%%%%%%%%%%%%%%%%%%%%%%%%%%%%%%%%%%%%%
%
\begin{proof}
\begin{enumerate}
\item The complex square root is, by definition, continuous on $\{ z \in \C
      : \Im (z) \leq 0 \}$ (cf. Definition~\ref{gen.eig}), hence the
      continuity of $K(x, x', \cdot)$ from below on the real axis. Note that
      $x$ and $x'$ are fixed parameters in this context.
\item For $  \mu = \lambda - i \eps$ and $x \in
      \overline{N_j}$ we have in concrete terms
     \[ K(x,x',\mu) = \frac{1}{w(\mu)}
         \left\{ \begin{aligned}
           &\bigl[\cos \bigl( \xi_j( \mu) x \bigr) - i
             s_j( \mu) \sin \bigl( \xi_j( \mu) x \bigr) \bigr]
             \exp(- i \xi_j(\mu)x'),
           && x' \in  \overline{N_j}, \ x'> x,\\
           & \exp \bigl( - i \xi_j(\mu)x \bigr)
             \bigl[ \cos \bigl( \xi_j( \mu) x' \bigr) - i
             s_j( \mu) \sin \bigl( \xi_j( \mu) x' \bigr) \bigr],
             && x' \in \overline{N_j}, \ x'< x,
             \\
           &\exp \bigl( - i \xi_j(\mu)x \bigr) \exp \bigl(- i \xi_k(\mu)x'
        \bigr),
           && x' \in \overline{N_k}, \ k \neq j.
         \end{aligned} \right.
      \]
    Now, let us first look at the case $\lambda > a_n$. Then
    \[ |\exp(-i \xi_j(\mu) x)| \le \exp \bigl( |\Im(\xi_j(\mu) x)| \bigr)
        = \exp \bigl( c_j^{-1/2} |\Im(\sqrt{\lambda - i \eps - a_j})|
        \bigr).
    \]
    Using the fact, that for $z \in \C$ with $|\arg(z)| \le \pi/2$ we have
    $|\Im(\sqrt{z})| \le \max \{1, |\Im(z)|\}$, we obtain
    \[ |\exp(-i \xi_j(\mu) x)| \le
        \exp \bigl( c_j^{-1/2} \max\{1,\delta\} x \bigr).
    \]
    In the case $a_1 \le \lambda \le a_n$ we find
    \[  |\xi_j(\mu)| = \sqrt{\frac{|\mu - a_j|}{c_j}} = c_j^{-1/2} \bigl(
        (\lambda - a_j)^2 + \eps^2 \bigr)^{1/4} \le c_j^{-1/2} \bigl(
        (a_n - a_j)^2 + \delta^2 \bigr)^{1/4}.
    \]
    Using these estimates and $|e^z| = e^{\Re(z)} \le e^{|z|}$, we find
    for all $\lambda \ge a_1$
        \begin{align*}
      |K(x,x',\mu)| &\le \frac{1}{|w(\mu)|}
            \left\{ \begin{aligned}
            & \bigl(1 + |s_j(\mu)| \bigr) \exp(|\xi_j(\mu)| x)
            \exp(|\xi_j(\mu)| x'), && x' \in  \overline{N_j}, \\
            &\exp(|\xi_j(\mu)| x) \exp(|\xi_k(\mu)| x'), && x' \in
            \overline{N_k}, k \neq j
            \end{aligned} \right. \\
      &\leq \frac{1}{|w(\mu)|} \bigl(1 + |s_j(\mu)| \bigr)
        \exp(\gamma (x+x')).
    \end{align*}
    The conclusion now follows using Lemma~\ref{w.esti} and Lemma~\ref{sj}.
    \qedhere
\end{enumerate}
\end{proof}
%%%
%%%
\noindent
Note that these estimates in particular imply that $N(\cdot,\delta) \in
L^1_{\mathrm{loc}}([a_1,+ \infty))$. In fact, if $a_1=\ldots=a_n$, then $M(\lambda,\delta)$ can be chosen to be constant, see Lemma~\ref{sj}, and the
denominator of $N$ causes only an $L^1_{\mathrm{loc}}$-type singularity. On the
other hand, if there are two different $a_j$, the denominator of $N$ is never
zero and $M(\cdot,\delta) \in L^1_{\mathrm{loc}}([a_1,+ \infty))$ again by
Lemma~\ref{sj}.
%%%%%%%%%%%%%%%%%%%%%%%%%%%%%%%%%%%%%%%%%%%%%%%%%%%%%%%%%%%%%%%%%%%%%%%%%%%%%
%
%
%%%%%%%%%%%%%%%%%%%%%%%%%%%%%%%%%%%%%%%%%%%%%%%%%%%%%%%%%%%%%%%%%%%%%%%%%%%%%
\begin{lem} \label{conjug}
For $(x,x') \in N^2$ and $\lambda \in \C$, it holds $K(x,x',
\overline{\lambda}) = \overline{K(x,x', \lambda)}$.
\end{lem}
%%%%%%%%%%%%%%%%%%%%%%%%%%%%%%%%%%%%%%%%%%%%%%%%%%%%%%%%%%%%%%%%%%%%%%%%%%%%%
%
\begin{proof}
The choice of the branch cut of the complex square root has been made such that
$\sqrt{\overline{\lambda}} = \overline{\sqrt{\lambda}}$ for all $\lambda \in \C$. This implies $\overline{e^{i \sqrt{\lambda}x}} = e^{\overline{i \sqrt{\lambda}x}} = e^{-i \sqrt{\overline{\lambda}}x}$
for all $\lambda \in \C$ and $x \in \R$. Thus it holds
\[ \overline{F_{\lambda}^{+,j}(x)} = F_{\overline{\lambda}}^{-,j}(x)
   \quad \text{and} \quad \overline{F_{\lambda}^{-,j}(x)} =
   F_{\overline{\lambda}}^{+,j}(x)
\]
for all $\lambda \in \C$, $x \in N$ and $j \in \{ 1, \dots, n \}$. In the same
way we have $\overline{w(\lambda)} = - w(\overline{\lambda})$. Observe,
that switching from $\lambda$ to $\overline{\lambda}$ the sign of the imaginary
part is changing, so in the definition of $K(x,x',\lambda)$ we have to take the
other sign, whenever there is a $\pm$-sign in the formula. This gives the assertion.
\end{proof}
%%%
%%%
Now, we can deduce a first formula for the resolution of the identity of
$A$.
%
%%%%%%%%%%%%%%%%%%%%%%%%%%%%%%%%%%%%%%%%%%%%%%%%%%%%%%%%%%%%%%%%%%%%%%%%%%%%%%
\begin{prop} \label{res.id}
Take $f \in H = \prod_{j=1}^n L^2(N_j)$, vanishing almost everywhere outside a
compact set $B \subset N$ and let $- \infty < a < b < + \infty$. Then for any continuous
scalar function $h$ defined on the real line and for all $x \in N$
\begin{align*}
& \bigl( h(A) E(a,b)f \bigr)(x) \\
\strut = & \int_{(a,b) \cap [a_1, + \infty)} h(\lambda)
          \sum\limits_{j=1}^{n}
          {\mathbf{1}}_{\overline{N_j}}(x)
          \Bigl\{
          \int\limits_{N} f(x')
          \Bigl[
          {\mathbf{1}}_{(x, +\infty)_{N_j}}(x') \cdot
          \Im \Big(\frac{1}{ w(\lambda)}
          F_{\lambda}^{-,j}(x) F_{\lambda}^{-,j+1}(x')\Big) \\
     & \strut \qquad + {\mathbf{1}}_{N \setminus ( x, +\infty )_{N_j}} (x')
    \cdot \Im \Big(\frac{1}{ w(\lambda)}
          F_{\lambda}^{-,j+1}(x) F_{\lambda}^{-,j}(x') \Big)
          \Bigr] \; dx'
          \Bigr\} \; d\lambda,
\end{align*}
where $E$ is the resolution of the identity of $A$ (cf.
Theorem~\ref{stone}) and the index $j$ is to be understood modulo
$n$, that is to say, if $j = n$, then $j + 1 = 1$.
\end{prop}
%%%%%%%%%%%%%%%%%%%%%%%%%%%%%%%%%%%%%%%%%%%%%%%%%%%%%%%%%%%%%%%%%%%%%%%%%%%%%%
\begin{proof}
The proof is analogous to that of Lemma~3.13 in \cite{fam2}. Let $g
\in H$ be vanishing outside $B$. Then
\begin{align}
   & \sprd {h(A) E(a,b) f},{g},H = \sprd{\lim_{\delta \rightarrow 0^{+}}
    \lim_{\varepsilon \rightarrow 0^{+}}
                \frac{1}{2 \pi i} \int_{a+\delta}^{b-\delta} h(\lambda)
                \bigl[ R(\lambda - \varepsilon i, A)
                - R(\lambda + \varepsilon i, A) \bigr]
                \; d \lambda \, f},
          {g},H \label{(2)} \\
  =& \lim_{\delta \rightarrow 0^{+}} \lim_{\varepsilon \rightarrow 0^{+}}
       \frac{1}{2 \pi i}
       \sprd{\int_{a+\delta}^{b-\delta} h(\lambda)
              \bigl[ R(\lambda - \varepsilon i,A)
              - R(\lambda + \varepsilon i, A) \bigr]
                \; d \lambda \, f},
            {g},H \label{(3)} \displaybreak[0]\\
  =& \lim_{\delta \rightarrow 0^{+}} \lim_{\varepsilon \rightarrow 0^{+}}
       \frac{1}{2 \pi i} \int_{a+\delta}^{b-\delta} h(\lambda)
       \sprd{\bigl[ R(\lambda - \varepsilon i, A)
                - R(\lambda + \varepsilon i, A) \bigr]
                f},
            {g},H \; d \lambda \label{(4)} \displaybreak[0]\\
=& \lim_{\delta, \varepsilon \rightarrow 0^{+}}
       \frac{1}{2 \pi i} \int_{(a+\delta,b-\delta) \cap [a_1, + \infty)}
    h(\lambda)  \sprd{\int_{N} f(x') \bigl[ K(\cdot,x',\lambda - i \varepsilon)
                - K(\cdot,x',\lambda + i \varepsilon) \bigr] \; dx'},
            {g(\cdot)},H \; d \lambda \label{(5)} \displaybreak[0] \\
  =& \lim_{\delta, \varepsilon \rightarrow 0^{+}}
        \frac{1}{2 \pi i} \int_{(a+\delta,b-\delta) \cap [a_1, + \infty)}
       h(\lambda) \sprd{\int_{N} f(x') \bigl[ K(\cdot,x',\lambda - i \varepsilon)-
                \overline{K(\cdot,x',\lambda - i \varepsilon)} \bigr]
                \; dx'},
             {g(\cdot)},H \; d \lambda \label{(6)} \displaybreak[0] \\
  =& \lim_{\delta, \varepsilon \rightarrow 0^{+}}
         \frac{1}{2 \pi i} \int_{ (a+\delta,b-\delta) \cap [a_1, + \infty)} h(\lambda)
         \sprd{\int_{N} f(x') \, 2 i \, \Im(
                K(\cdot,x',\lambda - i \varepsilon)) \; dx'},
              {g(\cdot)},H \; d \lambda \label{(7)} \displaybreak[0] \\
  =& \lim_{\delta \rightarrow 0^{+}} \frac{1}{\pi}
         \int_{(a+\delta,b-\delta) \cap [a_1, + \infty)} h(\lambda)
         \sprd{\int_{N} f(x') [ \lim_{\varepsilon \rightarrow 0^{+}}
                \Im(K(\cdot,x',\lambda - i \varepsilon)) ] \; dx'},
              {g(\cdot)},H \; d \lambda \label{(8)} \displaybreak[0] \\
  =& \sprd{\frac{1}{\pi} \int_{(a,b) \cap [a_1, + \infty)} h(\lambda)
 \int_N f(x') \Im(K(\cdot,x',\lambda - i 0)) \; dx' \; d \lambda},
          {g(\cdot)},H \label{(9)} \displaybreak[0] \\
 = & \int_{N} \frac{1}{\pi} \int_ {(a,b) \cap [a_1, + \infty)}
          h(\lambda) \Bigl\{ \int_{N} f(x') \Im \Bigl[  \frac{1}{ w(\lambda)} \sum_{j=1}^{n}
          {\mathbf{1}}_{\overline{N_j}}(x)
          \Big(
          {\mathbf{1}}_{( x, +\infty )_{N_j}}(x')
          F_{\lambda}^{-,j}(x) F_{\lambda}^{-,j+1}(x') \label{(10)}\\
  & \qquad \quad \strut + {\mathbf{1}}_{N \setminus ( x, +\infty )_{N_j}}(x')
    F_{\lambda}^{-,j+1}(x) F_{\lambda}^{-,j}(x') \Big)
          \Bigr] \; dx'
          \Bigr\} \;
          d\lambda \, g(x) \; dx. \nonumber
\end{align}
Here, the justifications for the equalities are the following:
\begin{itemize}
\item[\eqref{(2)}:] Stone's formula (Theorem~\ref{stone}).
\item[\eqref{(3)}:] After applying the operator valued integral to $f$, the
      two limits are in $H$. So they commute with the scalar product in
      $H$.
\item[\eqref{(4)}:] $\sprd {\cdot f},{g},H$ is a continuous linear form on
      $\lmd (H)$, and can therefore be commuted with the vector-valued
      integration. Note that $\lambda \mapsto R(\lambda,A)$ is continuous on
      the half-plane $\{ \lambda \in \C : \Re(\lambda) < a_1 \}$, since the
      resolvent is holomorphic outside the spectrum, cf. Proposition~\ref{selfadjoint}.
\item[\eqref{(5)}:] Theorem~\ref{theo1}.
\item[\eqref{(6)}:] Lemma~\ref{conjug}.
\item[\eqref{(7)}:] $z-\overline{z} = 2i \cdot \hbox{Im\,}z \ \forall z \in \C$.
\item[\eqref{(8)}:] Dominated convergence. Since supp$\, f$, supp$\, g$
     and $[a,b]$ are compact, we use the limiting absorption principle
     (Theorem~\ref{lim.abs}).
\item[\eqref{(9)}:] Fubini's Theorem.
\item[\eqref{(10)}:] Definition~\ref{def.K}.
\end{itemize}
The assertion follows, since $g$ was arbitrary with compact support.
\end{proof}
%%%
%%%
The unpleasant point about the formula in the above proposition is the apparent cut along the diagonal $\{x = x'\}$ expressed by the characteristic functions in the variable $x'$. In fact, there is no discontinuity and the
two integrals recombine with respect to $x'$. This is a consequence of the next lemma that gives an explicit representation of the integrand
above.

In the following, we use the convention
\[ a_{n+1} := + \infty,
\]
in order to unify notation and we set
\[ \xi_j' := i \xi_j.
\]
%
%%%%%%%%%%%%%%%%%%%%%%%%%%%%%%%%%%%%%%%%%%%%%%%%%%%%%%%%%%%%%%%%%%%%%%%%%%%%
\begin{lem} \label{l-calc.im}
Let $j,k \in \{ 1, \cdots, n \}$ and let $\lambda$ be fixed in
$(a_p,a_{p+1})$, with $p \in \{ 1, \cdots, n \}$.
Then $\Im \bigl[ \frac{1}{w} \bigl( F_\lambda^{-,j+1} \bigr)_j
(x) \bigl( F_\lambda^{-,j} \bigr)_k (x') \bigr]$ is given by the following
expressions, respectively:
\begin{itemize}
\item If $k \ge j > p$ or $j \geq k > p$ {\bf (Case (a))}
    \[ \Im \Bigl( \frac{1}{w} \Bigr) e^{-\xi'_j x - \xi'_k x'},
    \]
\item If $j < k \le p$ or $k < j \le p$ {\bf (Case (b), $j \neq k$)}
    \begin{align*}
      & \Im \Bigl( \frac{1}{w} \Bigr) \cos(\xi_j x) \cos(\xi_k x') - \Im
        \bigl( \frac{1}{w} \Bigr) \sin(\xi_j x) \sin(\xi_k x') \\
      & \qquad \strut - \Re \Bigl( \frac{1}{w} \Bigr) \cos(\xi_j x)
        \sin(\xi_k x') - \Re \Bigl( \frac{1}{w} \Bigr) \sin(\xi_j x)
        \cos(\xi_k x'),
    \end{align*}
\item If $j = k \leq p$ {\bf (Case (b), $j=k$)}
    \begin{align*}
      & \Im \Bigl( \frac{1}{w} \Bigr) \cos(\xi_j x) \cos(\xi_k x') - \Im
        \Bigl( \frac{s_k}{w} \Bigr) \sin(\xi_j x) \sin(\xi_k x') \\
      & \qquad \strut - \Re \Bigl( \frac{1}{w} \Bigr) \cos(\xi_j x)
        \sin(\xi_k x') - \Re \Bigl( \frac{1}{w} \Bigr) \sin(\xi_j x)
        \cos(\xi_k x'),
    \end{align*}
\item If $j \leq p < k$ {\bf (Case (c))}
    \[ \Im \Bigl( \frac{1}{w} \Bigr) e^{-\xi'_k x'} \cos(\xi_j x) + \Im
        \Bigl( \frac{1}{iw} \Bigr) e^{-\xi'_k x'} \sin(\xi_j x),
    \]
\item If $k \leq p < j$ {\bf (Case (d))}
    \[ \Im \Bigl( \frac{1}{w} \Bigr) e^{-\xi'_j x} \cos(\xi_k x') - \Im
        \Bigl( \frac{1}{iw} \Bigr) e^{-\xi'_j x} \sin(\xi_k x').
    \]
\end{itemize}
\end{lem}
%
%%%%%%%%%%%%%%%%%%%%%%%%%%%%%%%%%%%%%%%%%%%%%%%%%%%%%%%%%%%%%%%%%%%%%%%%%%%%
\begin{proof}
Since $\lambda$ belongs to $(a_p,a_{p+1})$, $\xi_j(\lambda)$ is a
real number, if and only if $j \leq p$. Otherwise, $\xi_j$ is a
purely imaginary number and we have $\xi_j = -i \xi'_j$.

Now, the proof is pure calculation, using the following expressions
for the generalized eigenfunctions in the case $j \neq k$
\begin{align*}
  \Re \bigl( F_{\lambda}^{-,j} \bigr)_k (x) &= \begin{cases}
        \cos(\xi_k x), & \text{if } \xi_k \in \R, \text{ i.e. } k \leq
            p, \\
        e^{- \xi'_k x}, & \text{if } \xi_k \in i \R, \text{ i.e. } k >
            p,
      \end{cases} \\
  \Im \bigl( F_{\lambda}^{-,j} \bigr)_k (x) &= \begin{cases}
        - \sin(\xi_k x), & \text{if } \xi_k \in \R, \text{ i.e. } k
            \leq p,  \\
        0, & \text{if } \xi_k \in i \R, \text{ i.e. } k > p,
    \end{cases}
\end{align*}
and for $j = k$
\begin{align*}
  & \Re \bigl( F_{\lambda}^{-,k} \bigr)_k (x) = \left\{ \begin{array}{ll} \cos(\xi_k x) + \Im(s_k)
    \sin(\xi_k x), & \text{if } \xi_k \in \R, \text{ i.e. } k \leq p, \\
    \Re \left( \frac{1}{2} (1 + s_k) \right) e^{- \xi'_k x} + \Re \left(
        \frac{1}{2} (1 - s_k) \right) e^{\xi'_k x}, & \text{if } \xi_k
        \in i \R, \text{ i.e. } k > p,
    \end{array} \right. \\
  &\Im \bigl( F_{\lambda}^{-,k} \bigr)_k (x) = \left\{ \begin{array}{ll}
    - \Re(s_k) \sin(\xi_k x), & \text{if } \xi_k \in \R, \text{ i.e. } k
        \leq p, \\
    \Im \left( \frac{1}{2} (1 + s_k) \right) e^{- \xi'_k x} + \Im \left(
        \frac{1}{2} (1 - s_k) \right) e^{\xi'_k x}, & \text{if } \xi_k
        \in i \R, \text{ i.e. } k > p,
    \end{array} \right.
\end{align*}
respectively.
\end{proof}
%%%
%%%
%
%%%%%%%%%%%%%%%%%%%%%%%%%%%%%%%%%%%%%%%%%%%%%%%%%%%%%%%%%%%%%%%%%%%%%%%%%%%%%
\begin{theo} \label{res.id_cycl}
Take $f \in H = \prod_{j=1}^n L^2(N_j)$, vanishing almost everywhere
outside a compact set $B \subset N$ and let $- \infty < a < b < +
\infty$. Then for any continuous scalar function $h$ defined on the
real line and for all $x \in N$
\begin{equation} \label{e-remrecomb}
\bigl( h(A) E(a,b)f \bigr)(x) =
          \int_{[a,b]\cap [a_1, +\infty)} \!\!\!h(\lambda)
          \sum_{j=1}^{n}
          {\mathbf{1}}_{\overline{N_j}}(x)
          \Bigl\{
          \int_N f(x')
          \Im \Bigl[\frac{1}{ w(\lambda)}
          F_{\lambda}^{-,j+1}(x) F_{\lambda}^{-,j}(x') \Bigr]
          dx'
          \Bigr\} \;
          d\lambda,
\end{equation}
where again the index $j$ is to be understood modulo $n$, i.e, if
$j = n$, then $j + 1 = 1$.
\end{theo}
%%%%%%%%%%%%%%%%%%%%%%%%%%%%%%%%%%%%%%%%%%%%%%%%%%%%%%%%%%%%%%%%%%%%%%%%%%%%%
%
\begin{proof}
All the work has already been done. It remains to inspect the formulae for $j =
k$ in the cases (a) and (b) of Lemma~\ref{l-calc.im}, to observe that the
expressions are symmetric in $x$ and $x'$. So for $x,x' \in N_j$ with $x < x'$
we find $F_\lambda^{-,j}(x) F_\lambda^{-,j+1}(x') = F_\lambda^{-,j}(x')
F_\lambda^{-,j+1}(x)$, which implies the assertion.
\end{proof}
%%%
%%%

%
%
%%%%%%%%%%%%%%%%%%%%%%%%%%%%%%%%%%%%%%%%%%%%%%%%%%%%%%%%%%%%%%%%%%%%%%%%%%%%

%%%%%%%%%%%%%%%%%%%%%%%%%%%%%%%%%%%%%%%%%%%%%%%%%%%%%%%%%%%%%%%%%%%%%%%%%%%%%
%
%
\section{Symmetrization} \label{spe.rep}
%
%
%%%%%%%%%%%%%%%%%%%%%%%%%%%%%%%%%%%%%%%%%%%%%%%%%%%%%%%%%%%%%%%%%%%%%%%%%%%%
%%%%%%%%%%%%%%%%%%%%%%%%%%%%%%%%%%%%%%%%%%%%%%%%%%%%%%%%%%%%%%%%%%%%%%%%%%%%
%
%
As was already explained in the introduction, the aim of this
section will be to find complex numbers $q_{l,m}$, $l,m \in \{1,
\dots, n\}$, such that the resolution of identity of $A$ can be
written as
\begin{equation} \label{specrepre}
  \left( E(a,b) f \right) (x) = \int_a^b \sum_{l,m = 1}^n q_{lm}(\lambda)
     F_\lambda^{-,l}(x) \int_N \overline{F_\lambda^{-,m}}(x')f(x') \; d x'
    \; d \lambda,
\end{equation}
in order to eliminate the cyclic structure of the formula in
Proposition~\ref{res.id_cycl}.

In this section we shall often suppress the dependence on $\lambda$
of several quantities for the ease of notation, so $s_j =
s_j(\lambda)$, $q_{l,m} = q_{l,m}(\lambda)$, $\xi_j =
\xi_j(\lambda)$, $\xi_j' = \xi_j'(\lambda)$ and $w = w(\lambda)$.
%
%%%%%%%%%%%%%%%%%%%%%%%%%%%%%%%%%%%%%%%%%%%%%%%%%%%%%%%%%%%%%%%%%%%%%%%%%%%%
\begin{lem} \label{lem-Leb}
Given $x \in N_j$, equation \eqref{specrepre} is satisfied for all
$a_1 \le a < b < + \infty$ and all $f \in L^2(N)$ with compact
support, if and only if for all $j = 1, \dots, n$
\begin{equation} \label{e-LebLem}
  \Im \Bigl[ \frac{1}{w} F_\lambda^{-,j+1} (x) F_\lambda^{-,j} (x') \Bigr] =
    \sum_{l,m = 1}^n q_{lm}(\lambda) F_\lambda^{-,l}\bigr (x)
    \overline{F_\lambda^{-,m}} (x')
\end{equation}
for almost all $x' \in N$ and $\lambda \ge a_1$. Here again the
index $j$ has to be understood modulo $n$, that is to say, if $j=n$,
then $j+1=1$.
\end{lem}
%%%%%%%%%%%%%%%%%%%%%%%%%%%%%%%%%%%%%%%%%%%%%%%%%%%%%%%%%%%%%%%%%%%%%%%%%%%
\begin{proof}
If functions $q_{l,m}$, $l,m = 1, \dots, n$, satisfy \eqref{e-LebLem} and if
$a_1 \le a < b < + \infty$, we get by \eqref{e-remrecomb}
\begin{align*}
(E(a,b)f)(x) &=
          \int\limits_{a}^{b}
          \sum\limits_{j=1}^{n}
          {\mathbf{1}}_{\overline{N_j}}(x)
          \Bigl\{
          \int\limits_{N} f(x')
    \sum_{l,m = 1}^n q_{lm}(\lambda) F_\lambda^{-,l}\bigr (x)
    \overline{F_\lambda^{-,m}} (x')
    \;dx'
          \Bigr\} \; d\lambda \\
  &= \int_a^b \sum_{l,m = 1}^n q_{lm}(\lambda)
     F_\lambda^{-,l}(x) \int_N \overline{F_\lambda^{-,m}}(x')f(x') \; d x'
    \; d \lambda,
\end{align*}
which is \eqref{specrepre}.

For the converse implication, let \eqref{specrepre} be satisfied for some $x
\in N_j$ and all $a_1 \le a < b < + \infty$, as well as all $f \in L^2(N)$ with
compact support. This means by \eqref{e-remrecomb}
\[ \int_{a}^{b}
\!
\int_{N} f(x') \Im \Bigl[\frac{1}{ w(\lambda)}
          F_{\lambda}^{-,j+1}(x) F_{\lambda}^{-,j}(x') \Bigr]
          dx'
%
%           \;
\!
           d\lambda = \int_a^b \sum_{l,m = 1}^n q_{lm}(\lambda)
     F_\lambda^{-,l}(x) \int_N \overline{F_\lambda^{-,m}}(x')f(x') \; d x'
%
%    \;
\!
    d \lambda.
\]
Firstly, we want to see that the integrands of the
$\lambda$-integrals on both sides in fact have to be equal almost everywhere.
In order to do so, we observe that they both are in $L^1((a_1, +\infty))$ and
use the following general observation: If $I$ is an interval and $g \in
L^1(I)$ satisfies $\int_J g = 0$ for all intervals $J \subseteq I$, then $g =
0$ almost everywhere in $I$. Indeed, in this case we have for any Lebesgue
point $x_0 \in I$ of $g$ and every $\eps > 0$ (cf. \cite[Theorem~8.8]{Rud})
\begin{align*}
  |g(x_0)| &= \biggl| \frac{1}{2 \eps} \int_{x_0 - \eps}^{x_0 + \eps} \bigl(
    g(x_0) - g(x) \bigr) \; d x + \frac{1}{2 \eps}
    \underbrace{\int_{x_0 - \eps}^{x_0 + \eps} g(x) \; d x}_{= 0}
    \biggr| \\
  &\le \frac{1}{2 \eps} \int_{x_0 - \eps}^{x_0 + \eps} \bigl| g(x) - g(x_0)
    \bigr| \; d x \longrightarrow 0 \quad (\eps \to 0),
\end{align*}
which implies $g(x_0) = 0$ for almost all $x_0 \in I$.

This implies
\[ \int_{N} f(x') \Im \Bigl[\frac{1}{ w(\lambda)}
          F_{\lambda}^{-,j+1}(x) F_{\lambda}^{-,j}(x') \Bigr]
          dx' = \int_N  \sum_{l,m = 1}^n q_{lm}(\lambda)
     F_\lambda^{-,l}(x) \overline{F_\lambda^{-,m}}(x')f(x') \; d x'
\]
for almost all $\lambda \ge a_1$ and all $f \in L^2(N)$ with compact
support. By the fundamental theorem of variational calculus this
implies the assertion.
\end{proof}
%%%
%%%
In a next step, we explicitely write down equation \eqref{e-LebLem} as a linear
system for the values $q_{lm}$.
%
%%%%%%%%%%%%%%%%%%%%%%%%%%%%%%%%%%%%%%%%%%%%%%%%%%%%%%%%%%%%%%%%%%%%%%%%%%%%
\begin{lem} \label{l-syst.qlm}
The equation
\[ \Im \Bigl[ \frac{1}{w} \bigl( F_\lambda^{-,j+1} \bigr)_j (x) \bigl(
    F_\lambda^{-,j} \bigr)_k (x') \Bigr] = \sum_{l,m = 1}^n
    q_{lm}(\lambda) \bigl( F_\lambda^{-,l}\bigr)_j (x) \bigl(
    \overline{F_\lambda^{-,m}} \bigr)_k (x')
\]
holds for any $(x,x') \in N_j \times N_k$ and $\lambda \in (a_p , a_{p+1})$,
if and only if
\begin{itemize}
\item {\bf Case (a)}: if $k \geq j > p$ or if $j \geq k >p$
    \[ \left \{ \begin{array}{rcl}
        q_{jk} &=& 0 \\
        \sum_{l \neq j} q_{lk} &=& 0\\
        \sum_{m \neq k} q_{jm} &=& 0\\
        \sum_{l \neq j, m \neq k} q_{lm} &=& \Im \left[ \frac{1}{w}
            \right],
      \end{array} \right.
    \]
\item {\bf Case (b), $j \neq k$}: if $j < k \leq p$ or if $k < j \leq p$
    \[ \left \{ \begin{array}{llllcl}
        &&&\sum_{l,m} q_{lm} &=& \Im \left[ \frac{1}{w} \right]  \\
        \sum_{l \neq j, m \neq k} q_{lm} &+ \overline{s_k}
            \sum_{l \neq j} q_{lk} &+ s_j \sum_{m \neq k} q_{jm} &+
            s_j \cdot \overline{s_k} \cdot q_{jk} &=& - \Im \left[
            \frac{1}{w} \right]  \\
        \sum_{l \neq j, m \neq k} q_{lm} &+ \overline{s_k}
            \sum_{l \neq j} q_{lk} &+ \sum_{m \neq k} q_{jm} &+
            \overline{s_k} \cdot q_{jk} &=& - i \cdot \Im \left[
            \frac{1}{iw} \right]  \\
        \sum_{l \neq j, m \neq k} q_{lm} &+ \sum_{l \neq j} q_{lk} &+
            s_j \sum_{m \neq k} q_{jm} &+ s_j \cdot q_{jk} &=& i
            \cdot \Im \left[ \frac{1}{iw} \right],
        \end{array} \right.
    \]
\item {\bf Case (b), $j=k$}: if $j = k \leq p$
    \[ \left \{ \begin{array}{llllcl}
        &&&\sum_{l,m} q_{lm} &=& \Im \left[ \frac{1}{w} \right]  \\
        \sum_{l \neq j, m \neq j} q_{lm} &+ \overline{s_j}
            \sum_{l \neq j} q_{lj} &+ s_j \sum_{m \neq j} q_{jm} &+
            s_j \cdot \overline{s_j} \cdot q_{jj} &=& - \Im \left[
            \frac{s_j}{w} \right] \\
        \sum_{l \neq j, m \neq j} q_{lm} &+ \overline{s_j}
            \sum_{l \neq j} q_{lj} &+ \sum_{m \neq j} q_{jm} &+
            \overline{s_j} \cdot q_{jj} &=& - i \cdot \Im \left[
            \frac{1}{iw} \right]  \\
        \sum_{l \neq j, m \neq j} q_{lm} &+ \sum_{l \neq j} q_{lj} &+
            s_j \sum_{m \neq j} q_{jm} &+ s_j \cdot q_{jj} &=& i
            \cdot \Im \left[ \frac{1}{iw} \right],
        \end{array} \right.
    \]
\item {\bf Case (c)}: if $j \leq p < k$
    \[ \left \{ \begin{array}{rcl}
        q_{jk} &=& 0 \\
        \sum_{l \neq j} q_{lk} &=& 0 \\
        \sum_{m \neq k} q_{jm} &=& \Im \left[ \frac{1}{w} \right] \\
        \sum_{l \neq j, m \neq k} q_{lm} + s_j \sum_{m \neq k} q_{jm}
            &=& i \cdot \Im \left[ \frac{1}{iw} \right],
    \end{array} \right.
    \]
\item {\bf Case (d)}: if $k \leq p < j$
    \[ \left\{ \begin{array}{rcl}
        q_{jk} &=& 0 \\
        \sum_{l \neq j} q_{lk} &=& \Im \left[ \frac{1}{w} \right] \\
        \sum_{m \neq k} q_{jm} &=& 0\\
        \sum_{l \neq j, m \neq k} q_{lm} + \overline{s_k}
        \sum_{l \neq j} q_{lk} &=& -i \cdot \Im \left[ \frac{1}{iw}
            \right].
    \end{array} \right.
    \]
\end{itemize}
\end{lem}
%%%%%%%%%%%%%%%%%%%%%%%%%%%%%%%%%%%%%%%%%%%%%%%%%%%%%%%%%%%%%%%%%%%%%%%%%%%%
%
\begin{proof}
The sum $\sum_{l,m = 1}^n q_{lm}(\lambda) \bigl( F_\lambda^{-,l}\bigr)_j (x) \bigl( \overline{F_\lambda^{-,m}} \bigr)_k (x')$
is explicitely written in the different cases (a), (b), (c) and (d) as it was done in Lemma~\ref{l-calc.im}. Then the linear
independence of the following families of functions is used to get the above systems for the $q_{lm}$'s:

\begin{itemize}
\item {\bf Case (a)}: $e^{A x + B x'}$, $e^{A x - B x'}$, $e^{-A x + B x'}$,
    $e^{-A x - B x'}$,

\medskip

\item {\bf Case (b)}: $\cos(A x) \cos(B x')$, $\cos(A x) \sin(B x')$,
    $\sin(A x) \cos(B x')$, $\sin(A x) \sin(B x')$,

\medskip

\item {\bf Cases (c) and (d)}: $e^{A x} \cos(B x')$, $e^{A x} \sin(B x')$,
    $e^{-A x} \cos(B x')$, $e^{-A x} \sin(B x')$,
\end{itemize}
where $A$ and $B$ are any fixed non-vanishing real numbers.

On the other hand: if the $q_{lm}$ satisfy the system indicated in the lemma, on both sides of the equation of the lemma we have the same linear combination
of the functions given above. Thus equality holds.
\end{proof}

Having the linear systems in the above lemma at hand, it remains to
solve them. Indeed this is possible and we briefly indicate the
necessary steps.

Firstly, if $\lambda < a_1$, we have $\xi_j \in i \R$ for all $j \in
\{ 1, \ldots, n \}$. Thus, for all $j,k \in \{ 1, \ldots, n \}$ and
all $(x,x') \in N_j \times N_k$
\[ \Im \Bigl[ \frac{1}{w} \bigl( F_\lambda^{-,j+1} \bigr)_j (x) \bigl(
    F_\lambda^{-,j} \bigr)_k (x') \Bigr] = \Im \Bigl( \frac{1}{w} \Bigr)
    e^{-\xi'_j x - \xi'_k x'} = 0.
\]

Now, let $\lambda$ be fixed in $(a_p,a_{p+1})$, with $p \in \{ 1, \cdots, n \}$, remembering the convention $a_{n+1}= + \infty$. Due to the first
equation of the corresponding system of cases (a), (c) and (d), the matrix $q
:= (q_{lm})_{l,m=1}^n$ has to be of the form
\[ q = \left( \begin{array}{cc}
    Q_p & \; \vline \; 0\\
    \hline
    0 & \; \vline \; 0
    \end{array} \right)
\]
with a $p \times p$ matrix $Q_p$.

Hence, the equations $\sum_{l \neq j} q_{lk} = 0$ and $\sum_{m \neq k}
q_{jm} = 0$ in case (a) are obviously fulfilled, as well as $\sum_{l \neq j}
q_{lk} = 0$ in case (c) and $\sum_{m \neq k} q_{jm} = 0$ in case (d). This
means that only three equations remain from the cases (a), (c) and (d):
\[ \left \{ \begin{array}{rcl}
    \sum_{l,m} q_{lm} &=& \Im \left[ \frac{1}{w} \right] \\
    \sum_{m \neq k} q_{lm} + \overline{s_k} \sum q_{lk} &=& - i \cdot \Im
        \left[ \frac{1}{iw} \right] \\
    \sum_{l \neq j} q_{lm} + s_j \sum q_{jm} &=& i \cdot \Im \left[
        \frac{1}{iw} \right],
    \end{array} \right.
\]
where in all the sums, $l$ and $m$ belong to $\{ 1, \ldots, p \}$.
Now it is important to note that these three equations are already
contained in the following system corresponding to the case (b),
i.e. the only conditions for the $q_{lm}$'s are the four following
equations for a fixed $(j,k)$ such that $j < k \leq p$ or $k < j
\leq p$:
\begin{equation} \label{e-DemQ:1}
  \left\{ \begin{array}{llllcl}
    &&&\sum_{l,m} q_{lm} &=& \Im \left[ \frac{1}{w} \right]  \\
    \sum_{l \neq j, m \neq k} q_{lm} &+ \overline{s_k} \sum_{l \neq j}
        q_{lk} &+ s_j \sum_{m \neq k} q_{jm} &+ s_j \cdot
        \overline{s_k} \cdot q_{jk} &=& - \Im \left[ \frac{1}{w}
        \right] \\
    \sum_{l \neq j, m \neq k} q_{lm} &+ \overline{s_k} \sum_{l \neq j}
        q_{lk} &+ \sum_{m \neq k} q_{jm} &+ \overline{s_k} \cdot q_{jk}
        &=& - i \cdot \Im \left[ \frac{1}{iw} \right] \\
    \sum_{l \neq j, m \neq k} q_{lm} &+ \sum_{l \neq j} q_{lk} &+ s_j
        \sum_{m \neq k} q_{jm} &+ s_j \cdot q_{jk} &=& i \cdot \Im
        \left[ \frac{1}{iw} \right]
    \end{array} \right.
\end{equation}
and, for $j = k \leq p$:
\begin{equation} \label{e-DemQ:2}
  \left \{ \begin{array}{llllcl}
    &&&\sum_{l,m} q_{lm} &=& \Im \left[ \frac{1}{w} \right]  \\
    \sum_{l \neq j, m \neq j} q_{lm} &+ \overline{s_j} \sum_{l \neq j}
        q_{lj} &+ s_j \sum_{m \neq j} q_{jm} &+ s_j \cdot
        \overline{s_j} \cdot q_{jj} &=& - \Im \left[ \frac{s_j}{w}
        \right] \\
    \sum_{l \neq j, m \neq j} q_{lm} &+ \overline{s_j} \sum_{l \neq j}
        q_{lj} &+ \sum_{m \neq j} q_{jm} &+ \overline{s_j} \cdot q_{jj}
        &=& - i \cdot \Im \left[ \frac{1}{iw} \right]  \\
    \sum_{l \neq j, m \neq j} q_{lm} &+ \sum_{l \neq j} q_{lj} &+ s_j
    \sum_{m \neq j} q_{jm} &+ s_j \cdot q_{jj} &=& i \cdot \Im \left[
    \frac{1}{iw} \right],
    \end{array} \right.
\end{equation}
where, once more, $l$ and $m$ belong to $\{ 1, \ldots, p \}$ in all
the sums.

Now, we denote by $Q_{jk} := \sum_{l \neq j} q_{lk}$ and $Q'_{jk}:=
\sum_{m \neq k} q_{jm}$. Using the fact that $\Im \left[ \frac{1}{w} \right] =
\Re \left[ \frac{1}{iw} \right]$, the above system \eqref{e-DemQ:1} is
\[ \left \{ \begin{array}{lrrcl}
    &\sum_{l \neq j, m \neq k} q_{lm} &&=& \Im \left[ \frac{1}{w} \right] -
        Q_{jk} - Q'_{jk} - q_{jk} \\
    (\overline{s_k} - 1) Q_{jk} + & (s_j - 1) Q'_{jk} +&
        (s_j \overline{s_k} - 1) q_{jk} &=& - 2 \Re \left[
        \frac{1}{iw} \right]  \\
    (\overline{s_k} - 1) Q_{jk} & +& (\overline{s_k} - 1)
        q_{jk} &=& - i \Im \left[ \frac{1}{iw} \right] - \Re \left[
        \frac{1}{iw} \right] = - \frac{1}{iw} \\
    & (s_j - 1) Q'_{jk} + & (s_j - 1) q_{jk} &=& i \Im \left[ \frac{1}{iw}
        \right] - \Re \left[ \frac{1}{iw} \right] =
        \frac{1}{i \overline{w}}.
\end{array} \right.
\]
Since $s_j - 1 = - \frac{i w}{c_j \xi_j}$, the last three equations may be
rewritten as
\[ \left \{ \begin{array}{rrrcl}
    (-i c_j \xi_j \overline{w}) Q_{jk} + & (i c_k \overline{\xi_k} w) Q'_{jk}
    + & (-i c_j \xi_j \overline{w} + i c_k \overline{\xi_k} w - |w|^2)
    q_{jk} &=& (2 c_j c_k \xi_j \overline{\xi_k}) \Re \left[ \frac{1}{iw}
    \right] \\
    Q_{jk} + && q_{jk} &=& \frac{c_k \overline{\xi_k}}{|w|^2} \\
    & Q'_{jk} + & q_{jk} &=& \frac{c_j \xi_j}{|w|^2} \\
\end{array} \right.
\]
and the Gauss method yields that $q_{jk}$ must vanish for $j \neq k$.

In the case $j=k$, we rewrite system \eqref{e-DemQ:2} as above and the three
equations to be solved turn out to be
\[ \left \{ \begin{array}{rrrcl}
    (-i c_j \xi_j \overline{w}) Q_{jj} + & (i c_j \overline{\xi_j} w) Q'_{jj}
    + & (-i c_j \xi_j \overline{w} + i c_j \overline{\xi_j} w - |w|^2)
        q_{jj} &=& (2 c_j^2 \xi_j \overline{\xi_j}) \left(\Re \left[
        \frac{1}{iw} \right] - \frac{1}{c_j\xi_j} \right)  \\
    Q_{jj} + &&  q_{jj} &=& \frac{c_j \overline{\xi_j}}{|w|^2} \\
    & Q'_{jj} + & q_{jj} &=& \frac{c_j \xi_j}{|w|^2} \\
\end{array} \right.
\]
The Gauss method once more gives the only possible solution $q_{jj}=
\frac{c_j \xi_j}{|w|^2}$ for any $j \in \{ 1, \cdots, p \}$.

With this candidate for a solution at hand, it is pure calculation to show
that it is indeed a solution. Thus we have shown the following result.
%
%%%%%%%%%%%%%%%%%%%%%%%%%%%%%%%%%%%%%%%%%%%%%%%%%%%%%%%%%%%%%%%%%%%%%%%%%%%%%%
\begin{theo} \label{maintheo}
Let $A$ be defined as in Section~\ref{sec2} and the generalized
eigenfunctions $F_\lambda^{-,j}$ be given by
Definition~\ref{gen.eig}. Take $f \in H = \prod_{j=1}^n L^2(N_j)$,
vanishing almost everywhere outside a compact set $B \subset N$, and
let $- \infty < a < b < + \infty$. Then for all $x \in N$
\[ \left( E(a,b) f \right) (x) = \int_a^b \sum_{l = 1}^n q_l(\lambda)
    F_\lambda^{-,l}(x) \int_N
    \overline{F_\lambda^{-,l}}(x')f(x') \; d x' \; d \lambda,
\]
where
\[ q_l(\lambda) := \begin{cases}
            0, & \text{if } \lambda < a_l, \\
                \frac{c_l(\lambda) \xi_l(\lambda)}{|w(\lambda)|^2}, &
        \text{if } a_l <
            \lambda.
                   \end{cases}
\]
Furthermore, for almost all $\lambda \in \R$ the matrix $q_{l,m}
(\lambda) = \delta_{lm} {\mathbf{1}}_{(a_l, +\infty )}(\lambda)
c_l(\lambda) \xi_l(\lambda)/|w(\lambda)|^2$ is the unique matrix satisfying
\eqref{specrepre}.
\end{theo}
%%%%%%%%%%%%%%%%%%%%%%%%%%%%%%%%%%%%%%%%%%%%%%%%%%%%%%%%%%%%%%%%%%%%%%%%%%%%%%
%
%
%
%%%%%%%%%%%%%%%%%%%%%%%%%%%%%%%%%%%%%%%%%%%%%%%%%%%%%%%%%%%%%%%%%%%%%%%%%%%%%
%%%%%%%%%%%%%%%%%%%%%%%%%%%%%%%%%%%%%%%%%%%%%%%%%%%%%%%%%%%%%%%%%%%%%%%%%%%%%
%
%
\section{A direct approach} \label{dir.app}
%
%
%%%%%%%%%%%%%%%%%%%%%%%%%%%%%%%%%%%%%%%%%%%%%%%%%%%%%%%%%%%%%%%%%%%%%%%%%%%%
%%%%%%%%%%%%%%%%%%%%%%%%%%%%%%%%%%%%%%%%%%%%%%%%%%%%%%%%%%%%%%%%%%%%%%%%%%%%
%
%
We have seen in the preceding section that a matrix $q(\lambda)$
satisfying \eqref{specrepre} exists and is unique up to a null set.
It is the aim of this section to deduce an alternative
representation for $q(\lambda)$, involving only $n \times
n$-matrices and not an $(3 n^2 + 1) \times n^2$ system as above.
Since this approach is essentially independent of the special
setting, it should be more convenient for generalizations.

In the following, let us consider the complex-valued generalized
eigenfunctions $F^{-,k}_{\lambda}$ as functions on $N$ and not as
elements of $\prod_{j=1}^n C^{\infty}(N_{j})$. We introduce the
notation

\[ F_{\lambda}(x) =
  \left(
    \begin{array}{c}
      F_\lambda^{-,1}(x) \\
      \vdots \\
      F_\lambda^{-,n}(x) \\
    \end{array}
  \right).
\]
Denoting by $e_k = (\delta_{lk})_{l=1,\dots,n} = (0, \dots, 0, 1, 0, \dots, 0)^T$ the $k$-th unit vector in $\C^n$, we set $d_1(\lambda) := F_\lambda(0)$
and for $j=2, \dots, n$ we fix $x_j \in N_j$ and set $d_j(\lambda) :=
F_\lambda(x_j)$. Due to the form of our generalized eigenfunctions, we then
have
\begin{align*} \label{basis}
   d_1(\lambda) &= \sum_{k=1}^n e_k = (1, \dots, 1)^T, \\
   d_j(\lambda) &= \beta_j e_j + \alpha_j \sum_{k \neq j} e_k = (\alpha_j,
    \dots, \alpha_j, \beta_j, \alpha_j, \dots, \alpha_j)^T \text{ for } j
    = 2, \dots, n
\end{align*}
for suitable $\alpha_j, \beta_j \in \C$.

Using these vectors, we now define
\[D(\lambda):= \sum_{j=1}^{n}d_{j}(\lambda)e_{j}^{T}.
\]
Since $\alpha_j \neq \beta_j$, for every $j \in \{ 2, \ldots, n \}$, by construction, the matrix $D$ is invertible for any choice
of $(x_2, \ldots, x_n)$ provided that $x_j \neq 0$ for all $j \in \{ 2, \ldots, n \}$. Indeed $\det D = \prod_{j=2}^n (\beta_j - \alpha_j)$. \\
Denoting by $C(\lambda)$ the diagonal matrix with $c_{11}=i$ and $c_{jj}= i
\alpha_j$ for any $j \in \{ 2, \ldots, n \}$, we can formulate our theorem.
%
%%%%%%%%%%%%%%%%%%%%%%%%%%%%%%%%%%%%%%%%%%%%%%%%%%%%%%%%%%%%%%%%%%%%%%
\begin{theo} \label{direct}
The matrix $q(\lambda):=(q_{lm}(\lambda))_{l,m=1,\ldots,n}$ satisfying
\eqref{specrepre} is given by
\begin{align*}
  q(\lambda) &= (D(\lambda)^{-1})^{T} \, \Im \Bigl(\dfrac{1}{w(\lambda)}\,
 \sum_{j=1}^{n}  e_j d_j(\lambda)^{T} e_{j+1} e_j^{T} D(\lambda) \Bigr)
 (\overline{D(\lambda)}^{-1})^{T} \\
&= (D(\lambda)^{-1})^{T} \, \Im \Bigl(\dfrac{-i}{w(\lambda)}\,
C(\lambda) D(\lambda) \Bigr) (\overline{D(\lambda)}^{-1})^{T}
\end{align*}
for almost all $\lambda > a_1$. As previously, $j$ has to be
understood modulo $n$, that is to say, $j+1=1$ if $j=n$.
\end{theo}
%%%%%%%%%%%%%%%%%%%%%%%%%%%%%%%%%%%%%%%%%%%%%%%%%%%%%%%%%%%%%%%%%%%%%%%%%%%%%
%
\begin{proof}
By Lemma~\ref{lem-Leb} the function $q$ satisfies \eqref{e-LebLem}. Using the
matrices and vectors introduced above, for fixed $j \in \{1,\ldots,n\}$ this
can be rewritten as
\[
\Im  \Bigl(\dfrac{1}{w(\lambda)}\, F_{\lambda} (x)^{T} e_{j+1}
e_j^{T} F_{\lambda} (x') \Bigr)
\ = \
F_{\lambda} (x)^{T} \ q(\lambda) \ \overline{F_{\lambda} }(x')
\]
for $x \in N_j$ and $x' \in N$.

Setting $x = x_j$ and $x' = x_k$ we obtain by the definition of
$d_j$
\[ \Im  \Bigl(\dfrac{1}{w(\lambda)}\, d_j (\lambda)^{T} e_{j+1}
    e_j^{T} d_k(\lambda) \Bigr) \ = \ d_j(\lambda)^{T} \, q(\lambda) \,
    \overline{d_k(\lambda)}
\]
for all $j,k \in \{1,\ldots,n\}$ and thus
\[\sum_{k=1}^{n} \sum_{j=1}^{n} e_j  \Im  \Bigl(\dfrac{1}{w(\lambda)}\, d_j (\lambda)^{T} e_{j+1}
e_j^{T} d_k(\lambda) \Bigr) e_k^T \ = \
\sum_{k=1}^{n}\sum_{j=1}^{n} e_j d_j(\lambda)^{T} \, q(\lambda) \, \overline{d_k(\lambda)} e_k^T.
\]
Using $\sum_{k=1}^{n}d_k e_k^T = D, $ we obtain
\[ \Im  \Bigl(\dfrac{1}{w(\lambda)}\, \sum_{j=1}^{n}  e_j d_j(\lambda)^{T} e_{j+1} e_j^{T}
 D \Bigr) \ = \ D^T \, q \, \overline{D}.
\]
Since $D$ is invertible and the relation $\sum_{j=1}^{n}  e_j d_j(\lambda)^{T}
e_{j+1} e_j^{T}= -i C(\lambda)$ holds true, we now get the desired formula for
$q$.
\end{proof}
%%%%%%%%%%%%%%%%%%%%%%%%%%%%%%%%%%%%%%%%%%%%%%%%%%%%%%%%%%%%%%%%%%%%%%%%%%%%%
%
%
%%%%%%%%%%%%%%%%%%%%%%%%%%%%%%%%%%%%%%%%%%%%%%%%%%%%%%%%%%%%%%%%%%%%%%%%%%%%%
%%%%%%%%%%%%%%%%%%%%%%%%%%%%%%%%%%%%%%%%%%%%%%%%%%%%%%%%%%%%%%%%%%%%%%%%%%%%%
%
%
\section{A Plancherel-type Theorem}\label{pla.the}
%
%
%%%%%%%%%%%%%%%%%%%%%%%%%%%%%%%%%%%%%%%%%%%%%%%%%%%%%%%%%%%%%%%%%%%%%%%%%%%%
%%%%%%%%%%%%%%%%%%%%%%%%%%%%%%%%%%%%%%%%%%%%%%%%%%%%%%%%%%%%%%%%%%%%%%%%%%%%
%
%
In this section we prove a Plancherel type theorem for a Fourier
type transformation $V$ and its inverse $V^{-1}$ associated with the
generalized eigenfunctions  $(F_{\lambda}^{-,k})_{k = 1, \ldots, n}$
introduced in Definition~\ref{gen.eig}. Due to the fact that $V$ and
$V^{-1}$ turn out to be not analogous, the surjectivity of $V$ will
not follow from its injectivity and needs a separate proof.

Furthermore, we show that the fact that a function $u$ belongs to
the space $D(A^k)$ can be formulated in terms of the decay rate of
$Vu$. These results should be useful for the resolution of evolution
problems involving the spatial operator $A$ as well as for the
analysis of the properties of a solution. For this part, we follow
Section~4 of \cite{fam2}.
%%%%%%%%%%%%%%%%%%%%%%%%%%%%%%%%%%%%%%%%%%%%%%%%%%%%%%%%%%%%%%%%%%%%%%%%%%%%%%
%
\begin{defi} \label{Lsigma2}
Let $q_k$ be the coefficients defined in Theorem~\ref{maintheo}. Considering
for every $k = 1, \dots, n$ the weighted $L^2$ space $L^2((a_k, +\infty),
q_k)$, we set $L^2_q := \prod_{k=1}^n L^2((a_k, + \infty),q_k)$. The
corresponding scalar product is denoted by
\[ (F, G)_q := \sum_{k=1}^n \int_{(a_k, +\infty)} q_k(\lambda) F_k(\lambda)
    \overline{G_k(\lambda)} \; d \lambda
\]
and the norm by $|F|_q := (F, F)_q^{1/2}$.
\end{defi}
\noindent
Note that $L_q^2$ is a Hilbert space, since it is the product of the Hilbert spaces $L^2((a_k, + \infty),q_k)$.

Now, we define the transformation $V$, together with its inverse $Z$, which is
later on shown to diagonalize $A$.
%
%%%%%%%%%%%%%%%%%%%%%%%%%%%%%%%%%%%%%%%%%%%%%%%%%%%%%%%%%%%%%%%%%%%%%%%%%%%%%
\begin{defi} \label{defV}
\begin{enumerate}
\item For all $f \in L^1(N, \C)$ and $k = 1, \ldots, n$ the function $(Vf)_k
    : [a_k , + \infty) \to \C$ is defined by
    \[ (Vf)_k(\lambda):= \int_N f(x) \overline{(F_{\lambda}^{-,k})(x)}
        \; dx.
    \]
\item We choose a cut-off function $\chi \in C^{\infty}(\R)$ such that $\chi    \equiv 0$ on $(- \infty, a_n + 1)$ and $\chi \equiv 1$ on $(a_n + 2 ,   + \infty)$. Then we set $X$ to be the space of all $(G_1, \dots, G_n )
    \in L^2_q$, such that $G_k \in C^{\infty}([a_k , + \infty), \C)$ and
    $\chi G_k$ can be extended by zero to an element of $\mathcal{S}(\R)$
    for all $k \in \{ 1, \ldots, n \}$.

    Now, for every $(G_1, \dots, G_n) \in X$ we define $Z(G_1, \ldots,
    G_n) : N \to \C$ by
    \[ Z(G_1, \ldots, G_n)(x):= \sum_{k=1}^n \int_{(a_k , + \infty )}
            q_k(\lambda) G_k(\lambda) (F_{\lambda}^{-,k})(x) \; d \lambda.
    \]
\end{enumerate}
\end{defi}
\begin{lem} \label{lemPlancherel} Let $f \in L^2(N)$ and $G = (G_1, \dots, G_n)
\in X$.
Then $G \in L_q^2$, $Z(G) \in L^2(N) = H$, $Vf = ((Vf)_1, \dots, (Vf)_n) \in
L_q^2$, and
\[(G,Vf)_q = (Z(G),f)_H.
\]
\end{lem}
%%%%%%%%%%%%%%%%%%%%%%%%%%%%%%%%%%%%%%%%%%%%%%%%%%%%%%%%%%%%%%%%%%%%%%%%%%%%%%
%
\begin{proof}
We have
\begin{align}
  (G,Vf)_q &= \sum_{k=1}^n \bigl(G_k, (Vf)_k \bigr)_{L^2((a_k, +\infty), q_k)}
    = \sum_{k=1}^n \int_{(a_k, +\infty)} q_k(\lambda) G_k(\lambda)
    \overline{(Vf)_k}(\lambda) \; d \lambda \nonumber \\
  &= \sum_{k=1}^n \int_{(a_k, +\infty)} q_k(\lambda) G_k(\lambda)
    \overline{\Bigl( \int_N f(x) \overline{(F_{\lambda}^{-,k})}(x) \; dx
    \Bigr)} \; d \lambda \nonumber \\
  &= \int_N \Bigl( \sum_{k=1}^n \int_{(a_k, +\infty)} q_k(\lambda)
    G_k(\lambda) \bigl(F_{\lambda}^{-,k} \bigr) (x) \; d \lambda \Bigr)
    \overline{f(x)} \; dx = (Z(G),f)_H. \label{fubini}
\end{align}
It only remains to make sure that the assumptions for Fubini's Theorem are satisfied to justify (\ref{fubini}). In fact, it is sufficient to estimate
$(\lambda,x) \mapsto q_k(\lambda) (F_{\lambda}^{-,k})(x)$ on $(a_k,a_{k+1})
\times N$ for a fixed $k \in \{1, \ldots, n \}$, since $f$ is compactly supported and $G_k$ is rapidly decreasing.

In order to do so, recall that $c_k \xi_k (\lambda) = \sqrt{c_k(\lambda -
a_k)}$ belongs to $\R$ if and only if $\lambda \geq a_k$ and to $i \R$
otherwise and that for any $\lambda > a_1$, we have $|w(\lambda)|^2 \ge
\sum_{l=1}^n c_l |\lambda - a_l|$ due to Lemma~\ref{w.esti}. Thus, putting in
the expression for $s_k(\lambda)$, cf. Definition~\ref{gen.eig}, we get for $x
\in N_k$
\begin{align}
  \bigl| q_k(\lambda) (F_{\lambda}^{-,k})(x) \bigr| &= \Bigl|
    \frac{c_k \xi_k (\lambda)}{|w(\lambda)|^2} \bigl( \cos(\xi_k
    (\lambda)x) - i s_k(\lambda) \sin(\xi_k (\lambda)x) \bigr) \Bigr| \le
    \frac{|c_k \xi_k (\lambda)|}{\sum_{l=1}^n c_l |\lambda - a_l|}
    (1 + |s_k(\lambda)|) \nonumber \\
  &\le \frac{|c_k \xi_k (\lambda)|}{\sum_{l=1}^n c_l |\lambda - a_l|} \cdot
    \frac{\sum_{l=1}^n c_l |\xi_l(\lambda)|}{|c_k \xi_k (\lambda)|} \le
    \frac{\sum_{l=1}^n \sqrt{c_l} \sqrt{|\lambda - a_l|}}{\sum_{l=1}^n c_l
    |\lambda - a_l|}. \label{e-est1}
\end{align}
Furthermore, for $x \in N_j$, $j > k$, we have
\begin{equation} \label{e-est2}
  |q_k(\lambda) (F_{\lambda}^{-,k})(x)| = \Bigl|
    \frac{c_k \xi_k (\lambda)}{|w(\lambda)|^2} e^{-i \xi_j(\lambda) x}
    \Bigr| \le \frac{1}{\sqrt{c_k} \sqrt{|\lambda - a_k|}} \exp \Bigl(
    -\sqrt{\frac{a_j - \lambda}{c_j}} x \Bigr).
\end{equation}
Finally, for $x \in N_j$, $j < k$,
\begin{equation} \label{e-est3}
  |q_k(\lambda) (F_{\lambda}^{-,k})(x)| = \Bigl|
    \frac{c_k \xi_k (\lambda)}{|w(\lambda)|^2} e^{-i \xi_j(\lambda) x}
    \Bigr| \le \frac{1}{\sqrt{c_k} \sqrt{|\lambda - a_k|}}.
\end{equation}
For all three cases the bound is an $L^1_{\mathrm{loc}}$ function of $\lambda$,
%on $[a_k,a_{k+1}]$ and so, belongs to $L^1_{\mathrm{loc}}([a_k,a_{k+1}])$,
which is enough for the application of Fubini's Theorem, due to the properties
of $f$ and $G$.
\end{proof}
%%%%%%%%%%%%%%%%%%%%%%%%%%%%%%%%%%%%%%%%%%%%%%%%%%%%%%%%%%%%%%%%%%%%%%%%%%%%%%
%
\begin{lem} \label{prop.Vf}
Consider $f \in \prod_{k=1}^n C_c^\infty(N_k)$. Then $((Vf)_1, \dots, (Vf)_n)
\in X$.
\end{lem}
%%%%%%%%%%%%%%%%%%%%%%%%%%%%%%%%%%%%%%%%%%%%%%%%%%%%%%%%%%%%%%%%%%%%%%%%%%%%
%
\begin{proof}
As in Section~4 of \cite{fam2}, $(Vf)_k$ is a linear combination
of Fourier and Laplace transforms of functions in
$C_c^\infty(N_j)$, $j \in \{ 1, \ldots, n \}$.
\end{proof}
%%%
%%%
The next step is to show that $V$ is an isometry and $Z$ is its
right inverse. The proof of the following theorem is analogous to that of
\cite[Theorem~4.9]{fam2}.
%
%%%%%%%%%%%%%%%%%%%%%%%%%%%%%%%%%%%%%%%%%%%%%%%%%%%%%%%%%%%%%%%%%%%%%%%%%%%%%%
\begin{theo} \label{V.iso}
Endow $\prod_{k=1}^n C_c^\infty(N_k)$ with the norm of $H = \prod_{k=1}^n
L^2(N_k)$. Then
\begin{enumerate}
\item  $V : \prod_{k=1}^n C_c^\infty(N_k) \to L^2_q$ is isometric and can
    therefore be extended to an isometry $\tilde{V} : H \to L^2_q$.

    \noindent In particular, for all $f \in H$ we have $|\tilde{V}f|_q^2 =
    (f,f)_H$.
\item $Z = \tilde{V}^{-1}$ on $\tilde{V}(H)$.
\item $Z$ can be extended to a continuous operator on $L^2_q$.
\end{enumerate}
\end{theo}
%%%%%%%%%%%%%%%%%%%%%%%%%%%%%%%%%%%%%%%%%%%%%%%%%%%%%%%%%%%%%%%%%%%%%%%%%%%%
%
In the sequel we shall again write $V$ for $\tilde{V}$ for simplicity. Note
that thanks to the density of $X$ in $L^2_q$, from Lemma~\ref{lemPlancherel}
we get the Plancherel type formula
\begin{equation} \label{e-Planch}
  (G,Vf)_q = (Z(G),f)_H \quad \text{for all } f \in L^2(N) \text{ and }
    G \in L^2_q.
\end{equation}
Set
\[ Y := \bigl\{(G_1, \dots, G_n) \in L^2_q : \mathrm{supp}(G_j)
    \text{ compact and } \mathrm{supp}(G_j) \cap \{ a_1, \dots, a_n \} =
    \emptyset, \ j = 1, \dots, n \bigr\}.
\]
Then $Y$ is a dense subspace of $L^2_q$ and for our subsequent considerations
it is important to know that also for all $G \in Y$ the transformation $Z$ is
given by the formula from Definition~\ref{defV}~ii). This is the aim of the
following lemma.
%
%%%%%%%%%%%%%%%%%%%%%%%%%%%%%%%%%%%%%%%%%%%%%%%%%%%%%%%%%%%%%%%%%%%%%%%%%%%%
\begin{lem} \label{lem-Y}
Let $G \in Y$. Then
\[ Z(G)(x) = \sum_{k=1}^n \int_{(a_k , + \infty )} q_k(\lambda) G_k(\lambda)    (F_{\lambda}^{-,k})(x) \; d \lambda.
\]
\end{lem}
%%%%%%%%%%%%%%%%%%%%%%%%%%%%%%%%%%%%%%%%%%%%%%%%%%%%%%%%%%%%%%%%%%%%%%%%%%%%%
%
\begin{proof}
Let $f \in \prod_{k=1}^n C_c^\infty(N_k)$. Then, using \eqref{e-Planch}, we get
\[ (Z(G), f)_H = (G, Vf)_q = \sum_{k=1}^n \int_{(a_k, +\infty)} G_k(\lambda)
    \int_N F_\lambda^{-,k}(x) f(x) \; dx \ q_k(\lambda) \; d \lambda.
\]
In order to apply Fubini's Theorem, we estimate for every $k = 1, \dots, n$
\begin{align*}
    \int_N \bigl| G_k(\lambda) F_\lambda^{-,k}(x) f(x) q_k(\lambda) \bigr|
    \; dx
  & = \bigl| G_k(\lambda) q_k(\lambda) \bigr| \int_N
    |F_\lambda^{-,k}(x) \bigr| |f(x)| \; dx \\
  & \le \  |G_k(\lambda)| \sup_{x \in \mathrm{supp}(f)} \bigl|
    q_k(\lambda) F_\lambda^{-,k}(x) \bigr| \cdot \|f\|_{L^1(N)}.
  \intertext{ Investing the estimates from \eqref{e-est1} -- \eqref{e-est3}, we
    can estimate this further by}
  & \le \  C \frac{1}{\sqrt{|\lambda - a_k|}} |G_k(\lambda)|\cdot \| f\|_{L^1(N)},
\end{align*}
and since $G \in Y$ this function is in $L^1((a_k, + \infty))$, so we may apply
the Fubini Theorem and obtain
\begin{align*}
  (Z(G), f)_H &= \int_N \sum_{k=1}^n \int_{(a_k, +\infty)}  q_k(\lambda)
    G_k(\lambda) F_\lambda^{-,k}(x) \; d\lambda \ f(x) \; d x \\
  &= \Bigl( \sum_{k=1}^n \int_{(a_k, + \infty)} q_k(\lambda) G_k(\lambda)
    F_\lambda^{-,k} \; d \lambda, f \Bigr)_H.
\end{align*}
This implies the assertion.
\end{proof}
We now want to show that $V$ diagonalizes the operator, i.e. $V(A^j u)
(\lambda) = \lambda^j V(u) (\lambda)$ for all $u \in D(A^j)$. In order to
formulate this precisely for a measurable function $\psi : \R \to \R$ we
denote the corresponding multiplication operator on $L^2_q$ by $M_\psi$,
i.e. for a function $F \in L^2_q$ we have
\[ (M_\psi F)_k(\lambda) = \psi(\lambda) F_k(\lambda), \qquad \lambda \in [a_k,
    +\infty), \ k \in \{1, \dots, n\}.
\]
Then we have the following lemma, whose proof is again analogous to
that of \cite[Lemma~4.12]{fam2}.
%%%%%%%%%%%%%%%%%%%%%%%%%%%%%%%%%%%%%%%%%%%%%%%%%%%%%%%%%%%%%%%%%%%%%%%%%%%%%%
\begin{lem} \label{mult.op}
\begin{enumerate}
\item Let $j \in \N$ and $p_j : \R \to \R$ be defined by $p_j(x) = x^j$. Then
    for any $f \in D(A^j)$ we have
    \[V (A^j f) = M_{p_j} Vf.
    \]
\item Let $\psi : [a_1,+ \infty) \rightarrow \R$ be a bounded, measurable
    function. Then $\psi(A)$ defined by the spectral Theorem is a bounded
    operator on $H$ and for all $f \in H$ we have
    \[ V (\psi (A) f) = M_{\psi} (Vf).
    \]
\end{enumerate}
\end{lem}
%%%%%%%%%%%%%%%%%%%%%%%%%%%%%%%%%%%%%%%%%%%%%%%%%%%%%%%%%%%%%%%%%%%%%%%%%%%%%
%
In order to show injectivity of $Z$, and thus the surjectivity of
$V$, we will need the following property of the kernel.
%
%%%%%%%%%%%%%%%%%%%%%%%%%%%%%%%%%%%%%%%%%%%%%%%%%%%%%%%%%%%%%%%%%%%%%%%%%%%%%
\begin{lem} \label{lem-mult}
Let $G \in \mathrm{ker}(Z)$ and let $\psi : [a_1, +\infty)$ be measurable and
bounded. Then $M_\psi G \in \mathrm{ker}(Z)$.
\end{lem}
%%%%%%%%%%%%%%%%%%%%%%%%%%%%%%%%%%%%%%%%%%%%%%%%%%%%%%%%%%%%%%%%%%%%%%%%%%%%%
%
\begin{proof}
Due to the boundedness of $\psi$, it holds $D(\psi(A)) = H$ and $M_\psi G \in
L^2_q$ for $G \in L^2_q$. Therefore, if $f \in L^2(N)$ with compact support and
$G \in \mathrm{ker}(Z)$, we have using \eqref{e-Planch} and Lemma~\ref{mult.op}
\begin{align*}
  (Z(M_\psi G), f )_H &= (M_\psi G, Vf)_{L^2_q} = (G, M_\psi Vf )_{L^2_q} =
    (G, V \psi(A) f)_{L^2_q} \\
  &= (Z(G), \psi(A) f)_H = (0, \psi(A) f)_H = 0 \, .
\end{align*}
\end{proof}
%%%
%%%
This lemma allows to reduce the proof of the injectivity for $Z$ to
the frequency bands $(a_k, a_{k+1})$ for $k = 1, \dots, n$ (set
again $a_{n+1} = + \infty$). There the injectivity of the  Fourier
sine transform will be used.
%
%%%%%%%%%%%%%%%%%%%%%%%%%%%%%%%%%%%%%%%%%%%%%%%%%%%%%%%%%%%%%%%%%%%%%%%%%%%%%%
\begin{theo}
$\mathrm{ker}(Z) = \{0\}$ and thus $V$ is a spectral representation of $H$
relative to $A$.
\end{theo}
%%%%%%%%%%%%%%%%%%%%%%%%%%%%%%%%%%%%%%%%%%%%%%%%%%%%%%%%%%%%%%%%%%%%%%%%%%%%%%
%
\begin{proof}
Let $G \in L^2_q$ with $Z(G) = 0$ be given. We fix $j \in \{ 1, \dots, n\}$ and
consider a compact interval $I \subseteq (a_j, a_{j+1})$.
Since $\mathbf{1}_I G$ lies in $Y$, invoking Lemmas~\ref{lem-mult}
and~\ref{lem-Y}, we obtain for almost all $x \in N$
\begin{equation} \label{e-Zinj}
  0 = Z(\mathbf{1}_I G)(x) = \sum_{k=1}^n \int_{(a_k, + \infty)}
    \!\!\!\!\!q_k(\lambda) \mathbf{1}_I(\lambda) G_k(\lambda)
    F_\lambda^{-,k}(x) \, d \lambda
%   &= \sum_{k=1}^j \int_{(a_k, a_{j+1})} q_k(\lambda)
%   \mathbf{1}_I(\lambda) G_k(\lambda) F_\lambda^{-,k}(x) \, d \lambda \\
  = \sum_{k=1}^j \int_I q_k(\lambda) G_k(\lambda) F_\lambda^{-,k}(x)
    \; d \lambda.
\end{equation}
In this equation, we now let $x$ to zero. This is possible, since the
hypotheses of Lebesgue's Theorem are satisfied thanks to the presence of $\mathbf{1}_I$. This implies thanks to $F_\lambda^{-,k}(0) = 1$
\[ 0 = \int_I \sum_{k=1}^j q_k(\lambda) G_k(\lambda) \; d \lambda.
\]
Since this is true for every compact interval $I \subseteq (a_j, a_{j+1})$
and since $q_k G_k \in L^1_{\mathrm{loc}}((a_j, a_{j+1}))$ for all $k = 1,
\dots, j$, we can argue as in the proof of Lemma~\ref{lem-Leb} to obtain
\begin{equation} \label{e-etape1}
  \sum_{k=1}^j q_k(\lambda) G_k(\lambda) = 0 \quad \text{for a.a. } \lambda \in
    (a_j, a_{j+1}).
\end{equation}
In a second step, we now consider $x \in N_\ell$ for some $\ell \le j$ in
\eqref{e-Zinj}. This leads to
\begin{align*}
  0 &= \sum_{{k=1}\atop{k \neq \ell}}^j \int_I q_k(\lambda) G_k(\lambda)
    \exp \bigl( - i \xi_\ell(\lambda) x \bigr) \; d \lambda + \int_I
    q_\ell(\lambda) G_\ell(\lambda) \bigl[ \cos \bigl( \xi_\ell(\lambda) x
    \bigr) - i s_\ell \sin \bigl( \xi_\ell(\lambda) x \bigr) \bigr] \;
    d \lambda \\
  &= \sum_{k=1}^j \int_I q_k(\lambda) G_k(\lambda) \exp \bigl( -i
    \xi_\ell(\lambda) x \bigr) \; d \lambda - i \int_I
    q_\ell(\lambda) G_\ell(\lambda) \bigl( s_\ell(\lambda) - 1 \bigr) \sin
    \bigl( \xi_\ell(\lambda) x \bigr) \; d \lambda.
\end{align*}
Thanks to \eqref{e-etape1}, the sum in the last expression vanishes, so we
are left with
\[ 0 = \int_I q_\ell(\lambda) G_\ell(\lambda) \bigl( s_\ell(\lambda) - 1
    \bigr) \sin \bigl( \xi_\ell(\lambda) x \bigr) \; d \lambda.
\]
In this case $\xi_\ell(\lambda) = \sqrt{(\lambda - a_\ell)/c_\ell}$ is real, so
we may substitute $p = \xi_\ell(\lambda)$ and obtain
\begin{align*}
  0 &= 2 c_\ell \int_{\xi_\ell(I)} q_\ell(a_\ell + c_\ell p^2) G_\ell(a_\ell +
    c_\ell p^2) \bigl[ s_\ell(a_\ell + c_\ell p^2) - 1 \bigr] \sin(px) p
    \; d p \\
  &= 2 c_\ell \int_{(0, +\infty)} \mathbf{1}_{\xi_\ell(I)} (p) q_\ell \bigl(
    \xi_\ell^{-1}(p) \bigr) G_\ell \bigl( \xi_\ell^{-1}(p) \bigr) \bigl[
    s_\ell \bigl( \xi_\ell^{-1}(p) \bigr) - 1 \bigr] p \sin(px) \; d p.
\end{align*}
This last integral is now the Fourier sine transform of
$\mathbf{1}_{\xi_\ell(I)} (q_\ell \circ \xi_\ell^{-1}) (G_\ell \circ
\xi_\ell^{-1}) \bigl[ (s_\ell \circ \xi_\ell^{-1}) - 1 \bigr]
\mathrm{id}$, so by the injectivity of this transformation we find
\[ 2 c_\ell q_\ell(\xi_\ell^{-1}(p)) G_\ell(
    \xi_\ell^{-1}(p)) \bigl[ (s_\ell(\xi_\ell^{-1}(p)) - 1 \bigr] p = 0
    \quad \text{for a.a. } p \in \xi_\ell(I)
\]
Dividing by $p$ and resubstituting this implies
\[ 2 c_\ell q_\ell(\lambda) G_\ell(\lambda) \bigl[ s_\ell(\lambda) - 1 \bigr]
    \xi_\ell(\lambda) = 0 \quad \text{for a.a. } \lambda \in I.
\]
Observing, that $q_\ell$, $\xi_\ell$ and $s_\ell - 1$ all may only vanish in
the points $a_1, \dots, a_n$, this finally implies $G_\ell(\lambda) = 0$ for
all $\ell \in \{1, \dots, j\}$ and almost all $\lambda \in (a_j, a_{j+1})$.
Since $j \in \{1, \dots, n\}$ was chosen arbitrarily, the assertion follows.
\end{proof}
%%%
%%%
Note that these considerations also provide a proof for the special case
$n=2$, whose proof in \cite[Theorem~4.17]{fam2} is not entirely correct.

\bigskip

Finally, we get the following characterization of the spaces $D(A^j)$.

%
%%%%%%%%%%%%%%%%%%%%%%%%%%%%%%%%%%%%%%%%%%%%%%%%%%%%%%%%%%%%%%%%%%%%%%%%%%%%%
\begin{theo} \label{DAj}
For $j \in \N$ the following statements are equivalent:
\begin{enumerate}
\item $u \in D(A^j)$,
\item $\lambda \mapsto \lambda^{j} (Vu)(\lambda) \in L_{q}^2$,
\item $\lambda \mapsto \lambda^{j} (Vu)_k(\lambda) \in
    L^2((a_k,+\infty),q_k)$, for all $k=1,\ldots,n$.
\end{enumerate}
\end{theo}
%%%%%%%%%%%%%%%%%%%%%%%%%%%%%%%%%%%%%%%%%%%%%%%%%%%%%%%%%%%%%%%%%%%%%%%%%%%%
\begin{proof}
Due to Theorem~\ref{V.iso}, it holds $(A^j u, A^j u)_H = |V (A^j u)
|_{q}^2$. Now Lemma~\ref{mult.op} implies
\[ (A^j u, A^j u)_H = |M_{p_j} Vu|_{q}^2.
\]
\end{proof}
%%%
%%%
The above results provide explicit solution formulae for evolution equations
involving the operator $A$. For example for the wave equation $\ddot u + A u =
0$ with $u(0)=u_0$ and $\dot u(0)=0$ a formal solution is given by $u(t)= Z
\cos (\sqrt{\lambda} t) V u_0$. Our aim in the future will be to study
properties of these solutions, as indicated in the introduction.
%
%%%%%%%%%%%%%%%%%%%%%%%%%%%%%%%%%%%%%%%%%%%%%%%%%%%%%%%%%%%%%%%%%%%%%%%%%%%%
%%%%%%%%%%%%%%%%%%%%%%%%%%%%%%%%%%%%%%%%%%%%%%%%%%%%%%%%%%%%%%%%%%%%%%%%%%%%
%
%
\section{Appendix}
%
%
%%%%%%%%%%%%%%%%%%%%%%%%%%%%%%%%%%%%%%%%%%%%%%%%%%%%%%%%%%%%%%%%%%%%%%%%%%%%
%%%%%%%%%%%%%%%%%%%%%%%%%%%%%%%%%%%%%%%%%%%%%%%%%%%%%%%%%%%%%%%%%%%%%%%%%%%%
%
%
\noindent
The book of J. Weidmann \cite{weid2} describes a general approach to
the spectral theory of systems of Sturm-Liouville equations, which is
in principle applicable to our setting. In this point of view, our problem
is seen as a system of $n$ equations on $(0,+\infty)$ coupled by boundary
conditions in $0$. Thus the kernel of the resolvent is a matrix-valued function
$\mathcal K(\cdot, \cdot, \lambda): (0,+\infty) \times (0,+\infty) \rightarrow
\C^{n \times n}$. The relation to Definition \ref{def.K} is given by
\[ \mathcal K_{ij}(x, x', \lambda) = K(x, x', \lambda) \text{ where } (x,x')
    \in N_i \times N_j \; \widehat{=} \; (0,+\infty) \times (0,+\infty).
\]
The fundamental hypothesis of Weidmann is that the kernel of the resolvent can
be written in the following form:
\begin{equation}\label{ker.weid}
\mathcal K(x,x',\lambda) \ = \ \left\{
                            \begin{array}{ll}
            \sum_{q=1}^{p}
                          \underbrace{
                          \Bigl(
                          \sum_{l=1}^{p} \alpha_{lq} \ w_{l}(x,\lambda)\Bigr)
                                       }_{ =: \  m_{q}^{\alpha}(x,\lambda)}
                                       w_{q}(x',\lambda)^{T}, & \hbox{ for } x' \leq x, \\
            \sum_{q=1}^{p}
                          \underbrace{
                          \Bigl(
                          \sum_{l=1}^{p} \beta_{lq} \ w_{l}(x,\lambda)\Bigr)
                                       }_{ =: \  m_{q}^{\beta}(x,\lambda)}
                                       w_{q}(x',\lambda)^{T}, & \hbox{ for } x' > x,
                            \end{array}
                          \right.
\end{equation}
where $\alpha_{lq}, \beta_{lq} \in \C$ and the $w_{q}: [0,\infty)
\times \mathbb{C} \rightarrow \mathbb{C}^n$ are such that $\{
w_q(\cdot,\lambda) : q = 1, \dots , p\}$ is a fundamental system of
$ \ker (A_{f} - \lambda I),$ the space of generalized eigenfunctions
of $A$. Here $A_{f}: D(A_{f}) \rightarrow C^0(\mathbb{R}) $ is the
formal operator, in our case
$A_{f} = A = (-c_k \cdot \partial^2_x + a_k)_{k=1, \dots, n}$
but $D(A_{f}) = \prod_{k=1}^{n}C^{2}(N_k)$, i.e. the operator $A$
without transmission conditions nor integrability conditions at
$\infty$. Clearly in our case $\dim \bigl( \ker (A_{f} - \lambda I
)\bigr) = 2n$ and thus
$p = 2n$.

In contrast to the typical applications treated in \cite{weid2} as
for instance the Dirac system,  we consider in this paper a \emph{transmission}
problem, i.e. intuitively the components of the $w_q$ are functions
on different domains $N_k$ (while mathematically all $N_k$ are equivalent
to $(0,+\infty)$): the branches of a star. For all such applications
it is highly important to use only generalized eigenfunctions
satisfying the transmission conditions $(T_0)$ and $(T_1)$, for
example Theorem \ref{DAj} would be impossible otherwise.

Supposing the ansatz \eqref{ker.weid} of Weidmann, this is not
possible, what we shall show now.

In Definition~\ref{def.K} we have given an explicit expression for the (unique)
kernel $K$ of the resolvent of $A$, using only elements of
$ \ \ker (A_T - \lambda I)$, where $A_T: D(A_T) \rightarrow H$
satisfies
$A_T = A = (-c_k \cdot \partial^2_x + a_k)_{k=1, \dots, n}$
and $D(A_T) = \prod_{k=1}^{n}C^{2}(N_k) \cap \{(u_k)_{k=1}^n$
satisfies $(T_0),(T_1) \}$. Note that $ \ker (A_T - \lambda I)$ is
the $n$-dimensional space of generalized eigenfunctions of $A$
satisfying the transmission conditions $(T_0)$ and $(T_1)$, but
without integrability conditions at infinity.

Suppose that we have a representation of $K$ as given in
\eqref{ker.weid} satisfying
$$w_q(\cdot,\lambda) \in \ker (A_T - \lambda I),\ q=1, \dots, p, \ \lambda \in \rho(A)$$
and thus we can take $p=n$. Let us fix $x \in N_1$ and $\lambda \in
\rho(A)$. Then the $m_{q}^\beta (x,\lambda) = : m_{q}^\beta \in
\mathbb{C}$ are constants and the expression
\begin{equation*}
    g(x',\lambda) := \sum_{q=1}^{p} m_{q}^\beta w_{q}(x',\lambda)
\end{equation*}
defines a function
\[ g(\cdot,\lambda) : \{ x' \in [ 0, \infty ) : x'  > x  \} \rightarrow
\mathbb{C}^n.
\]
Clearly, the functions
\[ [w_{q}(\cdot,\lambda)]_j: N_j \rightarrow \mathbb{C}
\]
can be uniquely extended to entire functions in $x'$ because they
are linear combinations of $e^{\pm i \xi_j (\lambda) x'}$. Thus the
$[g(\cdot,\lambda) ]_j $ are entire. Our hypothesis
$w_q(\cdot,\lambda) \in \ker (A_T - \lambda I),\ q=1, \dots, p$,
implies thus
\begin{equation}\label{transm.sat}
  g(\cdot,\lambda)\in \ker (A_T - \lambda I),\ q=1, \dots, p.
\end{equation}
But comparing \eqref{ker.weid} with Definition~\ref{def.K} yields
\begin{align*}
  g (x',\lambda) &= F_{\lambda}^{\pm,2}(x) \cdot F_{\lambda}^{\pm,1}(x')
    \text{ for } x' \in N_2, \dots , N_n
\intertext{and}
  g (x',\lambda) &= F_{\lambda}^{\pm,1}(x) \cdot F_{\lambda}^{\pm,2}(x')
    \text{ for } x' \in N_1,
\end{align*}
which makes sense after analytic continuation. Using the assumption
that $g$ satisfies $(T_0)$, we get from these two equalities,
putting $x' = 0 \in \overline{N_2}$ into the first and $x' = 0 \in
\overline{N_1}$ into the second
\[ F_\lambda^{\pm, 2}(x) = F_\lambda^{\pm, 2}(x) F_\lambda^{\pm, 1}(0)
    = g(0, \lambda) = F_\lambda^{\pm, 1}(x) F_\lambda^{\pm, 2}(0) =
    F_\lambda^{\pm, 1}(x).
\]
But inspecting the definitions of the generalized eigenfunctions,
bearing in mind that $x \in N_1$ was arbitrary, this implies
\[ \cos \bigl( \xi_1(\lambda) x \bigr) \pm i \sin \bigl( \xi_1(\lambda)
    x \bigr) = \cos \bigl( \xi_1(\lambda) x \bigr) \pm i s_1(\lambda)
    \sin \bigl( \xi_1(\lambda) x \bigr)
\]
and thus $s_1(\lambda) = 1$. This finally implies $\sum_{j=1}^n c_j
\xi_j(\lambda) = 0$, which is impossible for all real $\lambda <
a_1$, since then $\xi_j(\lambda)$ is purely imaginary. Furthermore,
$\sum_{j=1}^n c_j \xi_j$ is analytic, so this equality can, if ever,
only be fulfilled on a discrete set of $\lambda \in \C$.

We have thus proven:
%
%%%%%%%%%%%%%%%%%%%%%%%%%%%%%%%%%%%%%%%%%%%%%%%%%%%%%%%%%%%%%%%%%%%
\begin{theo} \label{comp.weidm}
Representation \eqref{ker.weid} of the kernel $\mathcal K$ of the
resolvent of $A$ is not possible, using exclusively generalized
eigenfunctions
\[ w_q(\cdot,\lambda) \in \ker (A_T - \lambda I),\ q=1, \dots, p.
\]
\end{theo}
%%%%%%%%%%%%%%%%%%%%%%%%%%%%%%%%%%%%%%%%%%%%%%%%%%%%%%%%%%%%%%%%%%%
%
The reason for this is, shortly speaking, the rigidity of
\eqref{ker.weid} caused by the use of the same linear combinations
of generalized eigenfunctions above the diagonals of all $N_j \times
N_k$ (idem below). This is not compatible with the fact that the
kernel is non-smooth only on the main diagonals, i.e. the diagonals
of $N_j \times N_j$, cf. Definition~\ref{def.K} and
Figure~\ref{NxN}.

This means that the approach of J. Weidmann in \cite{weid2}, when
applied to problems on the star-shaped domain, does not sufficiently
take into account its non-manifold character. The resulting
expansion formulae would use generalized eigenfunctions which are in
a sense incompatible with the geometry of the domain. This would
have undesirable consequences: for example, the important feature of
Theorem \ref{DAj} that the belonging of $u$ to $D(A^j)$ is expressed
by the decay of the components of $Vu$ would be impossible, due to
artificial singularities in the expansion formula.

%%%%%%%%%%%%%%%%%%%%%%%%%%%%%%%%%%%%%%%%%%%%%%%%%%%%%%%%%%%%%%%%
%                        BIBLIOGRAPHIE                         %
%%%%%%%%%%%%%%%%%%%%%%%%%%%%%%%%%%%%%%%%%%%%%%%%%%%%%%%%%%%%%%%%

\end{document}